\author{Wolfram Bauer and Joshua Isralowitz}

\documentclass[preprint, 12pt]{elsarticle}

\usepackage{amsthm} \usepackage{amssymb} \usepackage{amsfonts} \usepackage{amsmath}

\newtheorem{theorem}{Theorem}[section] \newtheorem{corollary}[theorem]{Corollary}
\newtheorem{proposition}[theorem]{Proposition} \newtheorem{lemma}[theorem]{Lemma}
\theoremstyle{remark} \newtheorem*{remark}{Remark} 

\DeclareMathOperator{\supp}{supp} \DeclareMathOperator{\card}{card}
\DeclareMathOperator{\diam}{diam} \DeclareMathOperator{\dist}{dist} \DeclareMathOperator{\s}{span}
\DeclareMathOperator{\SOT}{SOT} \DeclareMathOperator{\WOT}{WOT} \DeclareMathOperator{\e}{e}
\DeclareMathOperator{\Imag}{Im} 
\newcommand{\C}{\ensuremath{\mathbb{C}^n}} \newcommand{\B}{\ensuremath{B_\alpha }}
\newcommand{\Tf}{\ensuremath{T_f ^\alpha }} 
\newcommand{\Fp}{\ensuremath{F_\alpha ^p }} \newcommand{\Fq}{\ensuremath{F_\alpha ^q }}
\newcommand{\Ft}{\ensuremath{F_\alpha ^2 }} \newcommand{\Lt}{\ensuremath{L_\alpha ^2 }}
\newcommand{\Lp}{\ensuremath{L_\alpha ^p }}

\begin{document}

\begin{frontmatter}

\author{Wolfram Bauer and Joshua Isralowitz \fnref{1}} \fntext[1]{Both authors were supported by an
Emmy-Noether grant of Deutsche Forschungsgemeinschaft} \ead{wbauer@uni-math.gwdg.de,
jbi2@uni-math.gwdg.de} \address{Mathematisches Institut \\ Georg-August Universit\"{a}t
G\"{o}ttingen \\ Bunsenstra$\ss$e 3-5 \\ D-37073, G\"{o}ttingen \\ Germany}

\title{Compactness characterization of operators in the Toeplitz algebra of the Fock space $F_\alpha
^p$.}

\begin{abstract} For $1 < p < \infty$ let $\mathcal{T}_p ^\alpha$ be the norm closure of the algebra generated by Toeplitz operators
with bounded symbols acting on the standard weighted Fock space $F_\alpha ^p$.  In this paper, we will show that an operator $A$ is
compact on $F_\alpha ^p$ if and only if $A \in \mathcal{T}_p ^\alpha$ and the Berezin transform $B_\alpha
(A)$ of $A$  vanishes at infinity. \end{abstract}


\end{frontmatter}

\section{Introduction and preliminaries}

For any $\alpha >0$, consider the Gaussian measure \begin{equation*}
d\mu_{\alpha}(z):=\left(\frac{\alpha}{\pi}\right)^ne^{-\alpha |z|^2} dv(z), \end{equation*} where
$dv$ denotes the usual Lebesgue measure on $\mathbb{C}^n\cong \mathbb{R}^{2n}$. Let $1 \leq  p <
\infty$, and write $L_{\alpha}^p$ for the space of (equivalence classes) of measurable complex
valued function $f$ on $\mathbb{C}^n$ such that \begin{equation} \|f\|_{\alpha,p}:=\left[
\left(\frac{p\alpha}{2\pi}\right)^n\int_{\mathbb{C}^n}\Big{|}f(z)e^{-\alpha|z|^2/2}\Big{|}^pdv(z)\right]^{1/p}<\infty.
\tag{1.1} \end{equation} If $\mathcal{H}(\mathbb{C}^n)$ denotes the space of all entire functions on
$\mathbb{C}^n$, then the Fock space $F_{\alpha}^p$ is the Banach space defined by
$F_{\alpha}^p:=\mathcal{H}(\mathbb{C}^n)\cap L_{\alpha}^p$ with the norm $\|\cdot\|_{\alpha,p}$ (cf.
\cite{JPR}). Recall that $F_{\alpha}^2$ is a Hilbert space with the natural inner product $\langle
\cdot, \cdot \rangle _\alpha$ induced by $(1.1)$ and is sometimes called the Segal-Bargmann space,
cf. \cite{Barg}. In the case of  $p=\infty$, we define the Banach space $F_\alpha ^\infty$ by
\begin{equation*} F_{\alpha}^{\infty}:=\Big{\{}f\in \mathcal{H}(\mathbb{C}^n) \: : \:
\|f\|_{\alpha,\infty}: = \|fe^{-\frac{\alpha}{2}|\cdot|^2}\|_{\infty}<\infty\Big{\}}.
\end{equation*}

 Let $P_\alpha$ be the orthogonal projection from $L_\alpha ^2$ onto $F_\alpha ^2$ given by
 \begin{align} P_\alpha f (z) = \int_{\mathbb{C}^n} e^{\alpha (z \cdot \overline{w})} f(w) \,
 d\mu_\alpha (w). \nonumber \end{align}   It is well known \cite{JPR} that as an integral operator,
 $P_\alpha$ is a bounded projection from $L_\alpha ^p$ onto $  F_\alpha ^p$ for all $1 \leq p \leq
 \infty$.  If $g \in L^\infty$, then we can define the bounded Toeplitz operator $T_g ^\alpha$ on
 $F_\alpha ^p$ by the formula \begin{equation*} T_g ^\alpha := P_\alpha M_g \end{equation*} where
 $M_g$ is ``multiplication by $g$.''

Now let \begin{equation*} K^{\alpha} (w, z) = e^{\alpha (w \cdot \overline{z})} \end{equation*} be
the reproducing kernel of $F_\alpha ^2 $ and let $k_z ^{\alpha}$ be the corresponding normalized
reproducing kernel given by \begin{equation*}  k_z ^\alpha (w) = e^{\alpha (w \cdot \overline{z}) -
\frac{\alpha}{2} |z|^2}. \end{equation*} Given any bounded operator $A$ on $F_\alpha ^p$  for $1
\leq p < \infty$, let $B_\alpha (A)$ be the Berezin transform of $A$ defined by \begin{equation*}
 B_\alpha (A) (z) =  \langle A k_z ^\alpha, k_z ^\alpha \rangle _\alpha. \tag{1.2} \end{equation*}
 Since $\|k_z ^\alpha \|_{\alpha, p} = 1$ for all $z \in \mathbb{C}^n$ and $1 \leq p < \infty$, it
 is easy to see $ B_\alpha (A)$ is well defined and in fact bounded. Furthermore, it is well known and easy to show that the map $A \mapsto B_\alpha(A)$ is one-to-one if $A$ is bounded on $F_\alpha ^p$ (see \cite{F}, p. 42 for a proof of the special case $\alpha = 1$ and $p = 2$ that easily extends to general case $\alpha > 0$ and $1 < p < \infty$.) Moreover, for a ``nice enough''
 function $f$ (for example, $f \in L^\infty$), we define the Berezin transform of $f$ to be
 \begin{align*} B_\alpha (f)(z) & := \langle f k_z ^\alpha, k_z ^\alpha \rangle _\alpha \nonumber \\
 & = \left(\frac{\alpha}{\pi}\right) ^n \int_{\mathbb{C}^n} f(w)  e^{- \alpha |z - w|^2 } \, dv(w).
 \nonumber \end{align*}  Note that an easy application of Fubini's theorem gives us that $\B (\Tf) =
 \B (f)$. Also note that when $f$ is positive and measurable, we can define the (possibility
 infinite) function $B_\alpha (f)$ without any other assumptions on $f$.

Information regarding the operator $A$ can be often described in terms of properties of the function
$B_\alpha (A)$, and this point of view has been especially successful when dealing with the
boundedness, compactness, and Schatten class membership of $A$.  Note that $\underset{z \rightarrow
\infty}{\lim} \, k_z ^\alpha = 0$ weakly on $F_\alpha ^p$ if $1 < p < \infty$, so that $B_\alpha
(A)$ vanishes at infinity if $A$ is compact.  Unfortunately, the converse in general is not true for
bounded operators on the Fock space $F_\alpha ^p$, and is even not true for certain Toeplitz
operators on $F_\alpha ^p$  (see \cite{BC} for examples on $F_\alpha ^2$). However,  it was proven
in \cite{E} that $A$ is compact on $\Ft$ if and only if $B_\alpha (A)$ vanishes at infinity when $A$
is in the algebra generated by $\{T_f ^\alpha : f \in L^\infty\}$ (that is, $A$ is the finite sum of
finite products of Toeplitz operators $\Tf$ with bounded symbols $f$.)   Moreover, it was proved in
\cite{S} that if $1 < p < \infty$ and $A$  is any bounded operator on the standard unweighted Bergman space $L_a ^p (\mathbb{B}_n, dv)$ of the unit ball $\mathbb{B}_n$, then
$A$ is compact if and only if $A$ is in the norm closure of the algebra generated by Toeplitz
operators with bounded symbols and the Berezin transform $B (A)$ of $A$ associated to $L_a ^2
(\mathbb{B}_n, dv)$ vanishes on the boundary  $\partial \mathbb{B}_n$ (see also \cite{MW} where the results of \cite{S} are extended to the weighted Bergman space $L_a ^p (\mathbb{B}_n, dv_\gamma)$ with standard weights $dv_\gamma$ for $\gamma > -1$.)

 In this paper, we will show that this result also holds for the Fock space $F_\alpha ^p$ when $1 <
 p < \infty$. In particular, if $\mathcal{T}_p ^\alpha$ is the norm closure of the algebra generated by Toelitz operators with bounded symbols acting on $F_\alpha ^p$, then we will prove the following, which is the main result of this paper:

 \begin{theorem}  If $1 < p
  < \infty$ and $A$ is a bounded operator on $F_\alpha ^p$, then $A$ is compact if and only if $A
  \in \mathcal{T}_p ^\alpha$ and $B_\alpha (A)$ vanishes at infinity. \nonumber \end{theorem}

Moreover, in Section $3$, we will show that $\mathcal{T}_p ^\alpha$ is in fact the closed algebra generated by Toeplitz operators $T_\nu ^\alpha$ where $\nu$ is a complex Borel measure on $\C$ such that the total variation measure $|\nu|$ is Fock-Carleson, which greatly widens the scope of Theorem $1.1$ (see Section $2$ for a discussion of Fock-Carleson measures and Toeplitz operators with measure symbols.) Note that a version of this result was proven in \cite{S} for the unweighted Bergman space of the ball and was proven in \cite{MW} for the weighted Bergman space of the ball.
The basic strategy used to prove Theorem $1.1$ is similar to the strategy used in \cite{S} to prove
the corresponding result for the Bergman space, and in particular relies on obtaining quantitative
estimates for the essential norm $\|A\|_{\e}$ of operators $A \in \mathcal{T}_p  ^\alpha$ for $1 < p <
\infty$.  However, the details of the proofs involved in implementing this strategy will often be
considerably different than details found in \cite{S}.  Note that this is often the case when one is
trying to prove a result for the Fock space that is already known to be true for the Bergman space
(see \cite{BCI, E} for example).

We now give a short outline of the rest of the paper.  Sections $2 - 5$ will consist of preliminary
lemmas that will be used to prove Theorem $1.1$ and Section $6$ will contain a proof of Theorem
$1.1$.  More specifically, Section $2$ will discuss Fock-Carleson measures and prove an important
lemma (Lemma $2.6$) that will be used in Section $3$.  Section $3$ will discuss various
approximation results that will be needed to prove Theorem $1.1$.  Section $4$ will prove two
important lemmas (Lemmas $4.1$ and $4.3$) related to sampling and interpolation in Fock spaces.
Section $5$ will introduce a useful uniform algebra $\mathcal{A}$ and will extend the Berezin
transform and other related objects  that are defined on $\mathbb{C}^n$ to the maximal ideal space
$M_\mathcal{A}$ of $\mathcal{A}$. Section $6$ will tie all these ideas and results together to prove
Theorem $1.1$, and finally in Section $7$ we will discuss improvements to our results when $p  = 2$
that are similar to the ones in \cite{S}, and we will briefly discuss some open problems.

We will close this introduction with a short comment about the proof of Theorem $1.1$ and Section
$6$.  Since $(F_\alpha ^p)^* = F_\alpha ^q$ under the natural pairing induced by $F_\alpha ^2$  and
since $A \in \mathcal{T}_p ^\alpha$ if and only if $A^* \in \mathcal{T}_q ^\alpha$ where $q$ is the dual exponent
of $p$, it is easy to see that we only need to prove Theorem $1.1$ for $2 \leq p < \infty.$  More
generally, if $A \in \mathcal{T}_p ^\alpha$, then it will be seen later (using the two above mentioned
facts) that many of the necessary estimates for $\|A\|_{\e}$ when $A \in \mathcal{T}_p ^\alpha$ only need
to be obtained for the case $2 \leq p < \infty$.   This simple observation will be crucial to the
proof of Theorem $1.1$ since many of the needed preliminary estimates are only directly obtainable
for $p = 2$ and (in much weaker form) for $p = \infty$, and will subsequently follow for all $2 \leq
p < \infty$ either by duality or by complex interpolation.  Along these lines, we will often use the
following consequence of complex interpolation in Fock spaces (see \cite{JPR}): \begin{lemma} Let $2
< p < \infty$.  If $A$ is an operator $A : F_\alpha ^2 + F_\alpha ^\infty \rightarrow  L_\alpha ^2 +
L_\alpha ^\infty $ such that $A$ maps $\Ft$ to $\Lt$ boundedly and maps $F_\alpha ^\infty$ to
$L_\alpha ^\infty$ boundedly, then $A : \Fp \rightarrow \Lp$ boundedly.  More precisely, we have
that \begin{align} \|A\|_{F_\alpha ^p \rightarrow L_\alpha ^p} \leq \|A\|_{F_\alpha ^2 \rightarrow
L_\alpha ^2} ^\frac{2}{p} \|A\|_{F_\alpha ^\infty \rightarrow L_\alpha ^\infty} ^{1 - \frac{2}{p}}.
\nonumber \end{align} \end{lemma} \noindent

\section{Fock-Carleson measures and related operators}

Let $1 < p < \infty$.  A positive Borel measure $\nu$ on $\mathbb{C}^n$ will be called a
Fock-Carleson measure if \begin{equation*} \int_{\mathbb{C}^n} \left|f(z)  e^{-  \alpha |z|^2 /2 }
\right|^p \,   d\nu(z) \leq C \|f\|_{\alpha, p} ^p  \nonumber \end{equation*} for all $f \in
F_\alpha ^p$ with $C $ independent of $f$. In this case, define the Toeplitz operator $T_\nu ^\alpha
: F_\alpha ^p \rightarrow F_\alpha ^p$ by \begin{equation*} T_\nu ^\alpha f (z) =
\int_{\mathbb{C}^n} f(w) e^{\alpha (z \cdot \overline{w}) - \alpha |w|^2} \, d\nu(w). \nonumber
\end{equation*}   For any $r > 0$, and $z \in \mathbb{C}^n$, let $B(z, r)$ be a Euclidean ball
centered at $z$ with radius $r$. If $\nu$ is a Fock-Carleson measure, then we let $\imath_\nu $ be
the canonical imbedding from $F_\alpha ^p$ into $L^p _\alpha (d\nu)$ where $L^p _\alpha (d\nu)$ is the space of $\nu$ measurable functions $f$ where \begin{equation*} \int_{\mathbb{C}^n} \left|f(z)  e^{-  \alpha |z|^2 /2 }
\right|^p \,   d\nu(z)  < \infty. \end{equation*}  It turns out that the
property of $\nu$ being a Fock-Carleson measure is independent of both $p$ and $\alpha$, as the
following result in \cite{HL} shows:

\begin{lemma} For any $1 < p < \infty$ and any $\alpha,  r > 0$, the following quantities are
equivalent, where the constants of equivalence only depend on $p, n, \alpha$ and $r$:
\begin{list}{}{\setlength\parsep{0in}} \item $(a)  \ \|\nu\|_* := \underset{z \in
\mathbb{C}^n}{\sup} \int_{\mathbb{C}^n} e^{- \frac{\alpha}{2} |z - w|^2} \, d\nu(w), $ \item $(b) \
\|\imath_\nu\|_{F_\alpha ^p \rightarrow L^p _\alpha (d\nu)}^p,$ \item $(c) \ \underset{z \in
\mathbb{C}^n}{\sup} \nu(B(z, r)),$
 \item $(d) \  \|T_\nu ^\alpha \|_{F_\alpha ^p \rightarrow F_\alpha ^p} ^p.$ \end{list} \end{lemma} 

 Before we state and prove the main result of this section (Lemma $2.6$), we will need some
 preliminary results related to Fock-Carleson measures. Throughout the paper, we will let $C$ denote
 a constant that might change from estimate to estimate, or even from line to line in a single
 estimate.  We will indicate the parameters that $C$ depends on only when it is important to do so.

\begin{lemma} If $\nu$ is a Fock-Carleson measure and $F \subset \mathbb{C}^n$ is compact, then
there exists $C > 0$ independent of $f, F, $ and $\nu$  such that \begin{align} \|T_{\chi_F \nu}
^\alpha f\|_{F_\alpha ^2 \rightarrow F_\alpha ^2} \leq C \|\imath_\nu\|_{F_\alpha ^2 \rightarrow L^2
_\alpha (d\nu)} \left(\int_{F}  |f(z)|^2 e^{- \alpha |z|^2} \, d\nu(z) \right)^\frac{1}{2}
\nonumber \end{align} where $\chi_F$ is the characteristic function of $F$.\end{lemma}

\begin{proof} Let $g \in F_\alpha ^2$ with $\|g\|_{\alpha, 2} = 1$.  Using Fubini's theorem and the
fact that $\nu$ is a Fock-Carleson measure, we have that \begin{align} |\langle T_{\chi_F \nu}
^\alpha f, g \rangle_\alpha| & = \left(\frac{\alpha}{\pi}\right)^n \left| \int_{\mathbb{C}^n} \chi_F
(z) f(z) \overline{g(z)} e^{- \alpha |z|^2} \, d\nu(z) \right| \nonumber \\ & \leq
\left(\frac{\alpha}{\pi}\right)^n \|\imath_\nu\|_{F_\alpha ^2 \rightarrow L^2 _\alpha (d\nu)}
\left(\int_{F} |f(z)|^2 e^{- \alpha |z|^2} \, d\nu(z)\right)^\frac{1}{2}. \nonumber \end{align}
\end{proof}

\begin{lemma} For any $\alpha > 0$ and $s$ real, we have that \begin{align} \int_{\mathbb{C}^n}
\left| e^{s (z \cdot \overline{w})} \right| d\mu_\alpha(w) = e^{s^2 |z|^2 / 4\alpha}. \nonumber
\end{align}  \end{lemma}

\begin{proof} See \cite{DZ}. \end{proof}

\begin{lemma} For any $r, \alpha, p > 0$, any $z \in \mathbb{C}^n$, and any entire $f$, there exists
$C$ independent of $f$ and $z$ where \begin{align} \left|f(z) e^{- \frac{\alpha}{2} |z|^2}\right|^p
\leq C \int_{B(z, r)} \left|f(w) e^{- \frac{\alpha}{2} |w|^2}\right|^p \, dv(w). \nonumber
\end{align} \end{lemma}

\begin{proof} See \cite{IZ}. \end{proof}

For the rest of the paper, we will canonically treat $\mathbb{Z}^{2n}$ as a lattice in
$\mathbb{C}^n$. The following is the main technical result that is needed to prove the Lemma $2.6$.

\begin{lemma} Let $\nu$ be a Fock-Carleson measure and let $F_j, K_j \subset \mathbb{C}^n$ be Borel
sets where $\{F_j\}$ are pairwise disjoint and $d(F_j, K_j) > \delta \geq 1$ for each $j$. Then
\begin{align} \int_{\mathbb{C}^n} \sum_j [\chi_{F_j} (z) \chi_{K_j} (w) ] | e^{\alpha (z \cdot
\overline{w})} | e^{- \frac{\alpha}{2} |w|^2} \, d\nu(w) \leq o(\delta)  \|\nu\|_* e^{
\frac{\alpha}{2} |z|^2} \nonumber \end{align} where $o(\delta)$ only depends on $\delta$ (and
$\alpha, n$)  and $\underset{\delta \rightarrow \infty}{\lim} o(\delta) = 0$. \end{lemma}

\begin{proof} Clearly $K_j \subseteq \mathbb{C}^n \backslash B(z, \delta)$ if $z \in F_j$, which
means that \begin{align} \sum_j \chi_{F_j} (z) \chi_{K_j} (w) \leq \sum_j \chi_{F_j} (z)
\chi_{\mathbb{C}^n \backslash B(z, \delta)} (w). \nonumber \end{align}  Thus,  \begin{align}
\int_{\mathbb{C}^n} & \sum_j [\chi_{F_j} (z) \chi_{K_j} (w) ] | e^{\alpha (z \cdot \overline{w})} |
e^{- \frac{\alpha}{2} |w|^2} \, d\nu(w) \tag{2.1}  \\ & \leq \sum_j \chi_{F_j}(z) \int_{\mathbb{C}^n
\backslash B(z, \delta)} | e^{\alpha (z \cdot \overline{w})} | e^{- \frac{\alpha}{2} |w|^2} \,
d\nu(w) \nonumber \\ &= \sum_j \chi_{F_j} (z) J_z \nonumber \end{align} where \begin{align} J_z =
\int_{ B(z, \delta)^c} | e^{\alpha (z \cdot \overline{w})} | e^{- \frac{\alpha}{2} |w|^2} \, d\nu(w)
\nonumber \end{align}  and $B(z, \delta)^c = \mathbb{C}^n \backslash B(z, \delta)$. We will now
estimate $J_z$ using Lemma $2.1$ and Lemma $2.4$ with respect to the entire function $w \mapsto
e^{\alpha (w \cdot \overline{z})}$ :

\begin{align} J_z & \leq  \sum_{\begin{subarray}{c} \sigma \in \frac{1}{10} (2n)^{-1/2}
\mathbb{Z}^{2n} \\ B(\sigma, \frac{1}{10}) \cap B(z, \delta)^c \neq  \emptyset \end{subarray}}
\int_{B(\sigma, \frac{1}{10})}   | e^{\alpha (z \cdot \overline{w})} | e^{- \frac{\alpha}{2} |w|^2}
\, d\nu(w) \nonumber \\ & \leq C  \sum_{\begin{subarray}{c} \sigma \in \frac{1}{10} (2n)^{-1/2}
\mathbb{Z}^{2n} \\ B(\sigma, \frac{1}{10}) \cap B(z, \delta)^c \neq \emptyset \end{subarray}}
\int_{B(\sigma, \frac{1}{10})}   \int_{B(w, \frac{1}{10})} | e^{\alpha (z \cdot \overline{u})} |
e^{- \frac{\alpha}{2} |u|^2} \, dv(u) d\nu(w) \nonumber \\ & \leq C  \sum_{\begin{subarray}{c}
\sigma \in \frac{1}{10} (2n)^{-1/2} \mathbb{Z}^{2n} \\ B(\sigma, \frac{1}{10}) \cap B(z, \delta)^c
\neq \emptyset \end{subarray}} \nu(B(\sigma, \frac{1}{10})) \int_{B(\sigma, \frac{1}{5})} |
e^{\alpha (z \cdot \overline{u})} | e^{- \frac{\alpha}{2} |u|^2} \, dv(u) \nonumber \\ &  \leq C
\|\nu\|_*  \sum_{\begin{subarray}{c} \sigma \in \frac{1}{10} (2n)^{-1/2} \mathbb{Z}^{2n} \\
B(\sigma, \frac{1}{10}) \cap B(z, \delta)^c \neq \emptyset \end{subarray}}  \int_{B(\sigma,
\frac{1}{5})} | e^{\alpha (z \cdot \overline{u})} | e^{- \frac{\alpha}{2} |u|^2} \, dv(u). \nonumber
\end{align} Since $\delta \geq 1$ we have that \begin{align} B(\sigma, \frac{1}{10}) \backslash B(z, \delta) \neq \emptyset
\Longrightarrow B(\sigma, \frac{1}{5}) \cap B(z, \frac{\delta}{2}) = \emptyset \nonumber \end{align}
and since there exists $M > 0$ such that every $z \in \mathbb{C}^n$ belongs to at most $M$ of the
sets $B(\sigma, \frac{1}{5})$, we get that \begin{align} J_z  &\leq C \|\nu\|_* \int_{\mathbb{C}^n
\backslash B(z, \frac{\delta}{2})} | e^{\alpha (z \cdot \overline{u})} | e^{- \frac{\alpha}{2}
|u|^2} \, dv(u) \nonumber \\ & = C \|\nu\|_* \int_{\mathbb{C}^n \backslash B(0, \frac{\delta}{2})} |
e^{\alpha (z \cdot (\overline{z} - \overline{u}))} | e^{- \frac{\alpha}{2} |z - u|^2} \, dv(u)
\nonumber \\ & = C \|\nu\|_*  e^{\frac{\alpha}{2} |z|^2} \int_{\mathbb{C}^n \backslash B(0,
\frac{\delta}{2})} e^{-\frac{\alpha}{2} |u|^2} \, dv(u). \nonumber  \end{align} Finally, since the
sets $F_j$ are pairwise disjoint, we get that \begin{align} (2.1) & \leq C \|\nu\|_*
e^{\frac{\alpha}{2} |z|^2} o(\delta) \sum_j \chi_{F_j } (z) \nonumber \\ & \leq C \|\nu\|_*
e^{\frac{\alpha}{2} |z|^2} o(\delta)\nonumber \end{align} where \begin{align} o(\delta) :=
\int_{\mathbb{C}^n \backslash B(0, \frac{\delta}{2})} e^{-\frac{\alpha}{2} |u|^2} \, dv(u).
\nonumber \end{align} \end{proof}

\begin{lemma} Suppose that $\nu$ is a Fock-Carleson measure and that $F_j, K_j \subset
\mathbb{C}^n.$ Moreover, assume that $a_j \in L^\infty (dv)$ and $b_j \in L^\infty (d\nu)$ both with
norms $\leq 1$, and assume that \begin{list}{}{\setlength\parsep{0in}} \item $(a) \ d(F_j, K_j) >
\delta \geq 1,$ \item $(b) \  \supp a_j \subseteq F_j$ and $\supp b_j \subseteq K_j$, \item $(c) \
$every $z \in \mathbb{C}^n$ belongs to at most $N \in \mathbb{N}$ of the sets $F_j$. \end{list}

\noindent Then for $2 \leq p < \infty$, we have that $\sum_j M_{a_j} T_\nu ^\alpha  M_{b_j}$ is
bounded from $F_\alpha ^p $ to $L_\alpha ^p$ and   \begin{align} \|\sum_j M_{a_j} T_\nu ^\alpha
M_{b_j} \|_{F_\alpha ^p \rightarrow L_\alpha ^p} \leq N \|\nu\|_* o(\delta) \tag{2.2}\end{align}
where $o(\delta)$ only depends on $\delta$ (and $p, \alpha, n$)  and $\underset{\delta \rightarrow
\infty}{\lim} o(\delta) = 0$.  Moreover, for every $f \in F_\alpha ^p$ with norm $1,$ we have that
\begin{align} \sum_j \|M_{a_j} T_\nu ^\alpha M_{b_j} f\|_{ \alpha, p} ^p \leq N  \|\nu\|_* ^p
o(\delta) ^p. \tag{2.3} \end{align} \end{lemma}

\begin{proof} We will first prove the lemma for the special case $N = 1$.  Note that \begin{align} \left| \sum_j \left(M_{a_j} T_\nu ^\alpha M_{b_j} f
\right)(z) \right. &  \Bigg|   e^{- \frac{\alpha}{2} |z|^2}  \nonumber \\ & \leq   \sum_j |a_j(z)|
\int_{\mathbb{C}^n} |b_j(w)| |f(w)| e^{- \frac{\alpha}{2} |w|^2}  e^{- \frac{\alpha}{2} |z - w|^2}
d\nu(w)   \nonumber \\ & \leq  \|\nu\|_* \|f\|_{\alpha,  \infty}. \nonumber \end{align} Thus, by
Lemma $1.2$, the lemma will be proved for the special case $N = 1$ if we can show that $(2.2)$ holds (which is equivalent to $(2.3)$ when $N = 1$) when $p = 2$.

Moreover, by Lemma $2.1$, it is enough to prove that \begin{align} \|\sum_j M_{a_j} T_\nu ^\alpha
M_{b_j} \|_{ L_\alpha ^2(d\nu) \rightarrow L_\alpha ^2} \leq  \|\nu\|_* ^\frac{1}{2}  o(\delta)
\tag{2.2'}\end{align}  and \begin{align} \sum_j \|M_{a_j} T_\nu ^\alpha M_{b_j} f\|_{\alpha, 2} ^2
\leq   \|\nu\|_*  o(\delta) ^2 \tag{2.3'} \end{align} for every $f \in L_\alpha ^2 (d\nu)$ of norm
$\leq 1$.

Let \begin{align} \Phi(z, w) = \sum_j \chi_{F_j} (z) \chi_{K_j} (w)
|e^{\alpha (z \cdot \overline{w})}|  \nonumber \end{align} so that \begin{align} \left| \sum_j
\left(M_{a_j} T_\nu ^\alpha  M_{b_j} f \right)(z) \right| &= \left| \sum_j a_j(z)
\int_{\mathbb{C}^n} b_j(w) f(w) e^{\alpha (z \cdot \overline{w})} e^{ - \alpha |w|^2} d\nu(w)
\right| \nonumber \\ & \leq \int_{\mathbb{C}^n} \Phi(z, w) |f(w)| e^{-\alpha |w|^2} \, d\nu(w).
\nonumber \end{align} If $h(z) = e^{\frac{\alpha}{4} |z|^2}$, then Lemma $2.5$ tells us that
\begin{align} \int_{\mathbb{C}^n} \Phi(z, w) h(w)^2  e^{-\alpha |w|^2} \, d\nu(w) \leq o(\delta)
\|\nu\|_* h(z)^2. \nonumber \end{align} Moreover, Lemma $2.3$ tells us that \begin{align}
\int_{\mathbb{C}^n} \Phi(z, w) h(z)^2  d\mu_\alpha (z) & \leq C \int_{\mathbb{C}^n}  |e^{\alpha (z \cdot \overline{w})}| e^{- \frac{\alpha}{2} |z|^2} \, dv(z) =
C h(w)^2. \nonumber \end{align} The Schur test now proves $(2.2'), (2.3')$, and consequently proves the lemma for the special case $N = 1$.

 As in \cite{S}, the general case $N > 1$ follows easily from the special case when $N = 1$ by
 writing $\{F_j\}$ as the union of the family of sets $\{A_j ^i \}_{i = 1}^N$ where $\Lambda (z) =
 \{j : z \in F_j\}$, ordered in the natural way, and $A_j ^i = \{z \in F_j : j \text{ is the }
 i^{\text{th}} \text{ element of } \Lambda(z)\}$ (where $A_j ^i := \emptyset$ if $i >  \card
 \Lambda(z)$) so that $A_j ^i \cap A_k ^i = \emptyset$ for any $j \neq k$.

\end{proof}

\section{Approximation results for Fock space operators}

In this section, we will prove various approximation results that will be needed for the proof of
Theorem $1.1$. The first such result, Lemma $3.3$, will allow us to approximate operators of the
form $S T_\nu ^\alpha $ by simple sums of ``truncations'' of $S T_\nu ^\alpha$, where here $S \in
\mathcal{T}_p ^\alpha$ and $\nu$ is a Fock-Carleson measure.  For convenience, we will use the canonical
identification $\mathbb{C}^n \cong \mathbb{R}^{2n}$ and we will use the norm $|z|_\infty =
\max\{|z_1|, \ldots, |z_{2n}| \}$.  For some $\delta > 0$, enumerate the disjoint family of sets
$\{[- \delta, \delta)^{2n} + \sigma\}_{\sigma \in 2 \delta \mathbb{Z}^{2n}}$  as $\{B_j\}_{j =
1}^\infty$ and let $\Omega_\delta (B_j) = \{z \in \mathbb{C}^n : \dist_\infty (z, B_j) \leq \delta
\}$ where $\dist_\infty (z, B_j)$ is the distance between $z$ and $B_j$ in the $|\cdot|_\infty$
norm. The following result follows easily from the above definitions:

\begin{lemma}  For any $\delta > 0$, the Borel sets $B_j \subset \mathbb{C}^n$ above satisfy the
followig conditions: \begin{list}{}{\setlength\parsep{0in}} \item $(a) \ B_j \cap B_k = \emptyset $
if $j \neq k$, \item $(b) \ $Every $z \in \mathbb{C}^n$ belongs to at most $4^{2n}$ of the sets
$\Omega_\delta (B_j)$, \item $(c) \  \diam(B_j) \leq 2\delta \sqrt{2n} $ where $\diam(B_j)$ is the
Euclidean diameter of $B_j$. \end{list}

\end{lemma}

Now let $\delta > 0$ and $k$ be a non-negative integer.  Let $\{B_j\}_{j = 1}^\infty$ be a covering
of $\mathbb{C}^n$ satisfying the conditions of the above lemma for $(k + 1)\delta$ instead of
$\delta$. For $1 \leq i \leq k$ and $j \geq 1$, write \begin{align} F_{0, j} = B_j, \text{  and  }
F_{i + 1, j} = \{z \in \mathbb{C}^n : \dist_\infty (z, F_{i, j}) \leq \delta\}. \nonumber
\end{align} The next result is now easy to prove.

\begin{lemma} Let $\delta > 0$ and $k$ be a non-negative integer.  For each $1 \leq i \leq k + 1$
the family $\mathcal{F}^i = \{F_{i, j} : j \geq 1\}$ forms a covering of $\mathbb{C}^n$ such that

\begin{list}{}{\setlength\parsep{0in}} \item $(a) \ F_{0, j_1} \cap F_{0, j_2} = \emptyset$ if $j_1
\neq j_2$, \item $(b) \ F_{0, j} \subset F_{1, j} \subset \cdots \subset F_{k + 1, j}$ for all $j
\geq 1$, \item $(c) \ \dist_\infty (F_{i, j}, F_{i + 1, j} ^c) \geq \delta$ for all $0 \leq i \leq
k$ and $j \geq 1$, \item $(d) \ $Every point belongs to at most $4^{2n}$ elements of
$\mathcal{F}^i$, \item $(e) \ \diam(F_{i, j}) \leq \delta (3k + 1) \sqrt{2n}$ for each $i$ and $j.$

\end{list}

\end{lemma}

\begin{lemma} Let $2 \leq p < \infty,  S \in \mathcal{T}_p ^\alpha, \nu $ be a Fock-Carleson measure, and
let $\epsilon > 0$.  Then there are Borel sets $F_j \subset G_j \subset \mathbb{C}^n$ with $j \geq
1$ such that:

\begin{list}{}{\setlength\parsep{0in}} \item $(a) \ \mathbb{C}^n = \bigcup_{j \geq 1} F_j$, \item
$(b) \ F_j \cap F_k = \emptyset$ if $j \neq k$, \item $(c) \ $each point of $\mathbb{C}^n$ belongs
to at most $4^{2n}$ of the sets $G_j$, \item $(d) \ \diam G_j \leq d = d(S, \epsilon)$ and
\begin{align} \|S T_\nu ^\alpha - \sum_j M_{\chi_{F_j}} S T_{ \chi_{G_j} \nu} ^\alpha \|_{F_\alpha
^p \rightarrow L_\alpha ^p} \leq \epsilon. \nonumber \end{align} \end{list} \end{lemma}

\begin{proof} The proof is a combination of Lemmas $2.6$ and  $3.2$ and Lemmas $4.1$ and $4.2$ of
\cite{S}, and is identical to the proof of Theorem $4.3$ in \cite{S}.  In particular, $F_j := F_{0,
j}$ and $G_j := F_{k + 1, j}$ where $\{F_{i, j}\}$ are the sets that come from Lemma $3.2$ with
$\delta = \delta(S, \epsilon)$ set large enough to invoke the conclusion of Lemma $2.6$.
\end{proof}

We will now show that a Toeplitz operator $T_\nu ^\alpha $ where the total variation measure $|\nu|$ of $\nu$ is Fock-Carleson measure can be approximated in the $F_\alpha ^p$ operator norm for any $1 < p < \infty$ by Toeplitz operators $T_f ^\alpha$ with $f \in C_b ^\infty$, where as stated in the introduction, $C_b ^\infty$ is the space of $\mathbb{C}$ valued smooth functions where $f$ and all of its derivatives are bounded. In particular, this will imply that $ \mathcal{T}_p ^\alpha$ for any $1 < p < \infty$ is the closed algebra generated by $\{ T_\nu ^\alpha : |\nu| \text{ is Fock-Carleson} \}$.  First, however, we will need some preliminary definitions and results.

If $\nu$ is a complex Borel measure on $\C$ where $|\nu|$ is Fock-Carleson, then define the ``heat transform'' ${\widetilde{\nu}} ^{(t)}$ of $\nu$ at ``time'' $t>0$ to be \begin{equation*} { \widetilde{\nu}}^{(t)} (z):= \frac{1}{(4\pi
t)^n}\int_{\C} e^{-\frac{|w-z|^2}{4t}}\: d\nu(w). \end{equation*} If $f$ is a function on $\C$ such that $f\, dv$ is Fock-Carleson, then define $\widetilde{f} ^{(t)} := \widetilde{\nu}^{(t)}$ where $d\nu := f \, dv$.  A simple computation using Lemma $2.1$, Fubini's theorem, and the reproducing property gives us that \begin{equation*} B_\alpha (T_\nu ^\alpha) = \left(\frac{\pi}{\alpha}\right)^n \widetilde{\nu} ^{(\frac{1}{4\alpha})}. \end{equation*} Similarly, one can easily show that the semi-group property  \begin{equation*} \{\widetilde{\nu} ^{(s)} \tilde{\}}^{(t)}=\widetilde{\nu} ^{(s+t)} \end{equation*}  holds for $s,t>0$. Since Lemma $2.1$ says that
$\widetilde{|\nu|} ^{(t)}$ is bounded for all $ t > 0$, it follows easily from the semi-group property
that $\widetilde{\nu}^{(t)}$ is smooth and all of its derivatives are bounded.

Now for any $z \in \C$ and any complex Borel measure $\nu$, let $\nu_z$ be the complex Borel measure defined by $\nu_z (E) := \nu(z - E)$ for any Borel set $E \subset \C$.  Note that \begin{equation*} \int_{\C} f(z - w)  \, d\nu(w) = \int_{\C} f(w)  \, d\nu_z (w) \end{equation*}  for any $z \in \C$ and $f$ where $f(z - \cdot) \in L^1 (\C, d\nu)$.

\begin{lemma} If $\nu$ is a complex Borel measure such that $|\nu|$ is Fock-Carleson and $\nu_\beta := \widetilde{\nu}^{(\frac{1}{\beta})} \, dv - \nu$, then \begin{equation*} \lim_{\beta \rightarrow \infty} \, \sup_{z \in \C} \|B_\alpha (T_{(\nu_\beta)_z} ^\alpha   )\|_{\infty}  = 0 \end{equation*}  \end{lemma}

\begin{proof} From the discussion above, it is enough to show that \begin{equation*} \lim_{\beta \rightarrow \infty} \, \sup_{z \in \C} \|\widetilde{(\nu_\beta)_z} ^{(\frac{1}{4\alpha})}\|_{\infty}  = 0. \end{equation*} To that end, let $G_z := \widetilde{(\nu_z)} ^{(\frac{1}{8\alpha})}$ and note that Lemma $2.1$ gives us that $\sup_{z \in \C} \|G_z\|_{\infty} < C$ for some $C > 0$. Also note that \begin{equation*} (\nu_\beta)_z = (\nu_z)_\beta \end{equation*} for any $z \in \C$ and $\beta > 0$. 

Using the semi-group property and the above equality, we have that for $w \in \C$
\begin{align*} \left| \widetilde{(\nu_\beta)_z} ^{(\frac{1}{4\alpha})} (w) \right| & = \left| \widetilde{G_z} ^{(\frac{1}{\beta} + \frac{1}{8\alpha})} (w) -    \widetilde{G_z} ^{(\frac{1}{8\alpha})} (w) \right| \\ & = \frac{1}{\pi^n} \left|\int_{\C} G_z (u)  \left[ \left(\frac{2\beta\alpha}{\beta + 8 \alpha} \right)^n  e^{- \frac{2\beta\alpha}{\beta + 8 \alpha} |w - u|^2} - (2\alpha) ^n e^{- 2\alpha |w - u|^2} \right] \, dv(u) \right| \\ & \leq \frac{\|G_z \|_\infty}{\pi^n}  \int_{\C} \left|\left(\frac{2\beta\alpha}{\beta + 8 \alpha} \right)^n  e^{- \frac{2\beta\alpha}{\beta + 8 \alpha} |u|^2} - (2\alpha) ^n e^{- 2\alpha | u|^2}   \right| \, dv(u). \end{align*} The result now follows immediately by an application of the dominated convergence theorem.

\end{proof}

The proof of the following lemma is a slight variation of the proof of Lemma $3.4$ in \cite{NZZ}.

\begin{lemma} If $\nu$ is a complex Borel measure such that $|\nu|$ is Fock-Carleson then \begin{equation*} \lim_{\beta \rightarrow \infty} \, \sup_{z \in \C} | T_{(\nu_\beta)_z} ^\alpha 1 (w)| = 0 \end{equation*} where the convergence is pointwise for any $w \in \C$.
\end{lemma}

\begin{proof} First note that Lemma $2.1$ and the semi-group property tells us that $T_{(\nu_\beta)_z} ^\alpha$ is uniformly bounded in the $F_\alpha ^2 $ norm with respect to both $\beta$ and $z \in \C$.  Thus, by the reproducing property and an easy approximation argument, it is enough to show that \begin{equation*} \lim_{\beta \rightarrow \infty} \, \sup_{z \in \C} |\langle T_{(\nu_\beta)_z} ^\alpha 1, u^k \rangle_\alpha| = 0 \end{equation*} for each fixed multiindex $k \in \mathbb{N}_0 ^{n}$.

To that end, writing $k_w ^\alpha (u) = e^{- \frac{\alpha}{2} |w|^2} \sum_{\gamma} (\gamma !)^{-1} \alpha ^{|\gamma|} u^\gamma \overline{w}^\gamma$ and plugging this into the definition of the Berezin transform, we have that \begin{align*} & B_\alpha (T_{(\nu_\beta)_z} ^\alpha   ) (w)  \\ & =  e^{-\alpha |w|^2} \sum_{\gamma, \gamma'} \frac{\alpha ^{|\gamma|} \alpha ^{|\gamma'|}} {\gamma! \gamma' !} \langle T_{(\nu_\beta)_z} ^\alpha u^\gamma, u^{\gamma'} \rangle_\alpha w^\gamma \overline{w}^{\gamma'}. \end{align*} Then for any fixed multiindex $k$ and any $0 < r < 1$ we have \begin{align*} & \int_{B(0, r)} e^{\alpha |w|^2}  B_\alpha (T_{(\nu_\beta)_z} ^\alpha   ) (w)  w^k \, dv(w) \\ & = \sum_{\gamma, \gamma'} \frac{\alpha ^{|\gamma|} \alpha ^{|\gamma'|}} {\gamma! \gamma' !} \langle T_{(\nu_\beta)_z} ^\alpha u^\gamma, u^{\gamma'} \rangle_\alpha \int_{B(0, r)} w^{k + \gamma} \overline{w}^{\gamma'} \, dv(w)
 \\ & = r^{2n + 2|k|} \left( \frac{n!  \alpha ^{|k|}} {(n + |k|)! } \langle T_{(\nu_\beta)_z} ^\alpha 1 , u^k \rangle_\alpha + \sum_{|\gamma| = 1}^\infty \frac{n! \alpha ^{|\gamma|} \alpha ^{|k + \gamma|}} {(n + | k + \gamma|)!  \gamma! } \langle T_{(\nu_\beta)_z} ^\alpha u^\gamma , u^{\gamma + k} \rangle_\alpha r^{2 |\gamma|} \right). \end{align*} Since $T_{(\nu_\beta)_z} ^\alpha$ is uniformly bounded in the $F_\alpha ^2 $ operator norm, we have \begin{align*} &\left|\langle T_{(\nu_\beta)_z} ^\alpha 1 , u^k \rangle_\alpha \right| \\ & \leq  C r^{-2n -2 |k| }\|B_\alpha (T_{(\nu_\beta)_z} ^\alpha) \|_\infty \int_{B(0, r)} e^{\alpha |w|^2} |w^k| \, dv(w) + C \sum_{|\gamma| = 1}^\infty r^{2 |\gamma|} \end{align*} for some $C > 0$ independent of $z$ and $\beta$.  Lemma $3.4$ then gives us that \begin{equation*} \limsup_{\beta \rightarrow \infty} \, \sup_{z \in \C} \left|\langle T_{(\nu_\beta)_z} ^\alpha 1 , u^k \rangle_\alpha \right| \leq C \sum_{|\gamma| = 1}^\infty r^{2 |\gamma|} \end{equation*} and letting $r \rightarrow 0^+$ completes the proof.

\end{proof}

\begin{lemma} If $\nu$ is a complex Borel measure where $|\nu|$ is Fock-Carleson and $\nu_\beta$ is defined as in Lemma $3.4$, then there exists $C > 0$ independent of $\beta$ such that \begin{equation*}  \sup_{z \in \C} \left|T_{(\nu_\beta)_z} ^\alpha 1 (w)\right| \leq C e^{\frac{\alpha}{4} |w|^2} \end{equation*}

\end{lemma}

\begin{proof} Note that Lemma $2.1$ says that \begin{equation*}  \|\iota_{(\nu_\beta)_z}\|_{F_\alpha ^2 \rightarrow L_\alpha ^2 (d\nu)} < C \end{equation*} for some $C$ independent of $\beta$ and $z \in \C$ where $\iota_{\nu_z}$ is the canonical imbedding from $F_\alpha ^2$ into $L_\alpha ^2 (d{\nu_z})$. Thus, there exists $C > 0$ independent of $z \in \C$ such that \begin{align*}  \left|T_{\nu_z} ^\alpha 1 (w)\right| & \leq
\left(\frac{\alpha}{\pi}\right) ^n \int_{\C} | e^{\alpha (w \cdot \overline{u})}| e^{- \alpha |u|^2} \, d\nu_z (u) \\& \leq C \int_{\C} | e^{\alpha (w \cdot \overline{u})}| e^{- \alpha |u|^2} \, dv (u) \\ & = C e^{\frac{\alpha}{4} |w|^2} \end{align*} where the last equality comes from Lemma $2.3$.

\end{proof}

Finally, we can now prove \begin{theorem} If $\nu$ is a complex Borel measure where $|\nu|$ is Fock-Carleson and $1 < p < \infty$, then \begin{equation*} \lim_{\beta \rightarrow \infty} \|\left(\frac{\pi}{\alpha}\right)^{n} T_{ \widetilde{\nu} ^{(\frac{1}{\beta})}} ^\alpha  - T_{\nu} ^\alpha \|_{F_\alpha ^p \rightarrow F_\alpha ^p}  = 0. \end{equation*}  In particular, $\mathcal{T}_p ^\alpha$ for any $1 < p < \infty$ is the closed algebra generated by $\{ T_\nu ^\alpha : |\nu| \text{ is Fock-Carleson} \}$. \end{theorem}

\begin{proof} As before, let $\nu_\beta =  \widetilde{\nu}^{(\frac{1}{\beta})} \, dv - \nu$, so that \begin{equation*}  T_{\nu_\beta} ^\alpha = \left(\frac{\pi}{\alpha}\right)^{n} T_{ \widetilde{\nu} ^{(\frac{1}{\beta})}} ^\alpha -  T_{\nu} ^\alpha . \end{equation*} A direct calculation shows that \begin{equation*} \sup_\beta \|T_{\nu_\beta} ^\alpha \|_{F_\alpha ^\infty \rightarrow F_\alpha ^\infty} < \infty.  \end{equation*}  Thus, by an easy duality argument and Lemma $1.2$, it is enough to prove the theorem for $p = 2$.

To that end, we will proceed in a manner that is similar to the proof of Theorem $1$ in \cite{CIL}.  First note that
\begin{align*}  T_{\nu_\beta} ^\alpha f (w)  & = \langle T_{\nu_\beta} ^\alpha f, K^\alpha(\cdot, w) \rangle_\alpha   \\ & = \left(\frac{\alpha}{\pi}\right)^n  \int_{\C} f(w) \left(T_{\nu_\beta} ^\alpha K^\alpha (\cdot, u) \right)(w) e^{- \alpha |u|^2} \, dv(u) \end{align*} which means that $T_{\nu_\beta} ^\alpha$ is an integral operator on $F_\alpha ^2$ with kernel $\left(T_{\nu_\beta} ^\alpha K^\alpha (\cdot, u) \right)(w)$.  We will now use the Schur test to complete the proof.

Let $h(w) := e^{\frac{\alpha}{4} |w|^2}$ and let $\Phi_\beta (u, w) : = \left|\left(T_{\nu_\beta} ^\alpha K^\alpha (\cdot, u) \right)(w)\right|.$ By a simple change of variables we have that \begin{equation*} \Phi_\beta (u, w) = |e^{\alpha (w \cdot \overline{u})}| |T_{(\nu_\beta)_u} ^\alpha 1 (u - w)| \end{equation*} so that \begin{align*} \int_{\C} \Phi_\beta(u, w) h(u) ^2 \, e^{- \alpha |u|^2} \, dv(u) & = \int_{\C} |e^{\alpha (w \cdot \overline{u})}| |T_{(\nu_\beta)_u} ^\alpha 1 (u - w)|  e^{ -\frac{\alpha}{2} |u|^2} \, dv(u) \\ & = \int_{\C} |e^{\alpha (w \cdot \overline{(u + w)})}| |T_{(\nu_\beta)_{(u + w)}} ^\alpha 1 (u)|  e^{ - \frac{\alpha}{2} |u + w|^2} \, dv(u) \\ & =  e^{ \frac{\alpha}{2} |w|^2} \int_{\C}  |T_{(\nu_\beta)_{(u + w)}} ^\alpha 1 (u)|  e^{ - \frac{\alpha}{2} |u |^2} \, dv(u) \\ & \leq C e^{ \frac{\alpha}{2} |w|^2} \int_{\C}    e^{ - \frac{\alpha}{4} |u |^2} \, dv(u) \end{align*} where $C > 0$ comes from Lemma $3.6$.  Thus, we have that \begin{align*} \left(\frac{\alpha}{\pi}\right)^{n} \int_{\C} \Phi_\beta (u, w) h(u) ^2 \, e^{- \alpha |u|^2} \, dv(u) \leq C h(w) ^2 \end{align*} for some $C > 0$ independent of $w$ and $\beta$.

Furthermore, an application of the Cauchy-Schwarz inequality gives us that \begin{align*} & \int_{\C} \Phi_\beta(w, u) h(u) ^2 \, e^{- \alpha |u|^2} \, dv(u)  \\ &  = \int_{\C} |e^{\alpha (u \cdot \overline{w})}| |T_{(\nu_\beta)_w} ^\alpha 1 (w - u)|  e^{ -\frac{\alpha}{2} |u|^2} \, dv(u) \\ &  = \int_{\C} |e^{\alpha ((w - u) \cdot \overline{w})}| |T_{(\nu_\beta)_w} ^\alpha 1 (u)|  e^{ -\frac{\alpha}{2} |w - u|^2} \, dv(u) \\ & \leq \left(\int_{\C} |T_{(\nu_\beta)_w} ^\alpha 1 (u)|^2 e^{- \frac{4\alpha}{5} |u|^2} \, dv(u) \right)^\frac{1}{2} \left(\int_{\C} e^{ \frac{4\alpha}{5} |u|^2}  |e^{2\alpha ((w - u) \cdot \overline{w})}| e^{ -\alpha |w - u|^2} \, dv(u) \right)^\frac{1}{2} \end{align*}

\noindent By a simple computation, we have that   \begin{equation*} \left(\int_{\C} e^{ \frac{4\alpha}{5} |u|^2}  |e^{2\alpha ((w - u) \cdot \overline{w})}| e^{ -\alpha |w - u|^2} \, dv(u) \right)^\frac{1}{2} = e^{\frac{\alpha}{2} |w|^2} \left(\int_{\C} e^{-\frac{\alpha}{5} |u|^2} \, dv(u) \right)^\frac{1}{2}. \end{equation*}

However,  Lemmas $3.5$ and $3.6$ give us that \begin{equation*} \lim_{\beta \rightarrow \infty} \, \sup_{w \in \C} |T_{(\nu_\beta)_w} ^\alpha 1 (u)|^2 = 0 \end{equation*} pointwise in $u$ and \begin{equation*} \sup_{w \in \C} |T_{(\nu_\beta)_w} ^\alpha 1 (u)|^2 e^{- \frac{4\alpha}{5} |u|^2}  \leq C e^{- \frac{3\alpha}{10} |u|^2} \end{equation*} for some $C > 0$ independent of $\beta$.  Thus, the dominated convergence theorem gives us that \begin{equation*} \int_{\C} \Phi_\beta (w, u) h(u) ^2 \, e^{- \alpha |u|^2} \, dv(u) \leq c(\beta) h(w)^2 \end{equation*} where $c(\beta)$ is independent of $w$ and $\lim_{\beta \rightarrow \infty} c(\beta) = 0$.  An application of the Schur test now completes the proof.

\end{proof}

Next we will show that all compact operators $A$ on $F_\alpha ^p $ for $1 < p < \infty$ are
contained in the Toeplitz algebra $\mathcal{T}_p ^\alpha$. Note that this was first proved in \cite{E1} (using completely different methods) for the special case $p = 2$ and $\alpha = \frac{1}{2}.$   For any $f \in \Fp$ and $g \in \Fq$ where $q$ is the dual exponent of $p$,
let $f \otimes g$ be the standard tensor product operator on $\Fp$ defined by \begin{equation*} f
\otimes g = \langle \cdot, g \rangle_\alpha f. \end{equation*}  Since all $L^p$ spaces have the
bounded approximation property (see \cite{W}, p. 69-70), this will be proved (by linearity) if we can show that
each finite rank operator $f \otimes g$ on $F_\alpha ^p$ can be approximated in the operator norm by a Toeplitz
operator with symbol in $C_b ^\infty$. First we will show that $1 \otimes 1$ can be approximated in the operator
norm by a Toeplitz operator with symbol in $C_b ^\infty$.

\begin{lemma} Let $\beta > 0$ and let $1  < p < \infty.$  If \begin{equation*}
q_{\beta}(z):=\left(\frac{\beta}{\pi}\right)^n \exp\left\{ -\beta|z|^2\right\},
\end{equation*} then \begin{equation*} \underset{\beta \rightarrow \infty}{\lim}
\|T_{q_\beta}^{\alpha} - 1 \otimes 1\|_{F_\alpha ^p \rightarrow F_\alpha ^p} = 0. \nonumber
\end{equation*} \end{lemma}

\begin{proof}  Note that $(1 \otimes 1) h =  T^\alpha _{\delta_0} h = h(0)$ if $h \in F_\alpha
^p$ where $\delta_0$ is the usual point-mass measure at $ 0\in \mathbb{C}^n$.  Also, note that by defintion we have $\widetilde{\delta_0} ^{(\frac{1}{4\beta})} = q_\beta$ for each $\beta > 0$.  The result now immediately follows from Theorem $3.7$ since $\delta_0$ is a Fock-Carleson measure.  \end{proof}

Given $w \in \mathbb{C}^n$, define the ``weighted shift'' operator $C_{\alpha}(w)$ on $L_\alpha ^p$
by \begin{equation*} \left[C_{\alpha}(w)f\right](z):=f(z-w)e^{\alpha (z\cdot
\overline{w})-\frac{\alpha}{2}|w|^2}. \end{equation*}
 It is known (see \cite{JPR}) that $C_{\alpha}(w)$ is
an isometry of $F_{\alpha}^p$ (and $L_{\alpha}^p$) onto itself for $1\leq p\leq \infty$, and it is
easy to check that $C_{\alpha}(w)^{-1}=C_{\alpha}(-w)$. Moreover, for $w_1,w_2\in \mathbb{C}^n$, one
has that \begin{equation*}\label{Composition_rule_C_alpha}
C_{\alpha}(w_1)C_{\alpha}(w_2)=e^{-i\alpha \textup{\scriptsize Im} (w_1\cdot
\overline{w_2})}C_{\alpha}(w_1+w_2). \end{equation*}

Note that the operators $C_{\alpha}(w)$ are in fact Toeplitz operators with bounded symbols. To see
this, if $s_w(z):=\exp\{ \frac{\alpha}{2}|w|^2+2i\alpha \:\textup{Im} (z\cdot \overline{w})\}$, then
we have that \begin{align*} \left[ T_{s_w}^{\alpha}f\right](u)
&=e^{\frac{\alpha}{2}|w|^2}\Big{\langle} fK_{\alpha}(\cdot,
w)\overline{K_{\alpha}(\cdot,-w)},K_{\alpha}(\cdot,u)\Big{\rangle}_{\alpha}\\
&=e^{\frac{\alpha}{2}|w|^2}\Big{\langle} fK_{\alpha}(\cdot,
w),K_{\alpha}(\cdot,u-w)\Big{\rangle}_{\alpha}\\ &=f(u-w)K_{\alpha}(u-w,w)e^{\frac{\alpha}{2}|w|^2}
\nonumber \\ & =\Big{[}C_{\alpha}(w)f\Big{]}(u) \end{align*} for all $f \in \Fp \cap \Ft$, which
shows that $T_{s_w}^{\alpha}=C_{\alpha}(w)$ on $\Fp.$

 Since $C_{\alpha}(v) T_{q_{\beta}(\cdot -u)}^{\alpha}$ can be written as the Toeplitz product $
 T_{s_v}^{\alpha} T_{q_{\beta}(\cdot -u)}^{\alpha}$, it can be shown that $C_{\alpha}(v)
 T_{q_{\beta}(\cdot -u)}^{\alpha} = T_{f_\beta}^\alpha $ where $f_\beta : =s_v \; \sharp_{\alpha}\;
 q_{\beta}(\cdot-u)$ and the
``product'' $\sharp_{\alpha}$ is defined by \begin{equation*} \psi\: \sharp_{\alpha}\: \varphi :=
\sum_{\gamma\in \mathbb{N}_0^n}\frac{1}{(-\alpha)^{|\gamma|}\gamma!}
\frac{\partial^{|\gamma|}\psi}{\partial z^{\gamma}}\cdot \frac{\partial^{|\gamma|}\varphi}{\partial
\overline{z}^{\gamma}} \end{equation*} for suitable smooth functions $\psi$ and $\varphi$ on
$\mathbb{C}^n$ (see \cite{Bau} for more details.) Using this formula, one can directly compute that
\begin{equation*} f_\beta(z)=\left(\frac{\beta+\alpha}{\alpha}\right)^n\exp\Big{\{}
\frac{\alpha}{2}|v|^2+\beta(z-u)\cdot \overline{v}+2i\alpha\: \textup{Im}( z\cdot
\overline{v})-\beta|z-u|^2\:\Big{\}}. \end{equation*} Note that one could also directly verify the
equality  $C_{\alpha}(v)T_{q_{\beta}(\cdot -u)}^{\alpha}=T_{f_\beta} ^{\alpha}$ where $f_\beta$ is
defined as above by comparing the Berezin transforms of both sides.

Using these shift operators and their properties, we can now prove \begin{theorem} If $1 < p <
\infty, f \in F_\alpha ^p,$ and $g \in F_\alpha ^q$, then  $f \otimes g \in \mathcal{T}_p ^\alpha.$ \end{theorem}

\begin{proof} Since $\s \{K(\cdot, w) : w \in \mathbb{C}^n\}$ is dense in $F_\alpha ^p$, it is
enough to show that each $K(\cdot, v) \otimes K(\cdot, w)  $ is in $\mathcal{T}_p ^\alpha$.
Furthermore, if $g \in F_\alpha ^p$, then \begin{align*} C_{\alpha}(v)\big{(} 1\otimes 1
\big{)}C_{\alpha}(-w)g &=\left\langle C_{\alpha}(-w)g,1\right\rangle_{\alpha}C_{\alpha}(v)1\\
&=\left\langle g,C_{\alpha}(w)1\right\rangle_{\alpha} C_{\alpha}(v)1\\
&=e^{-\frac{\alpha}{2}(|w|^2+|v|^2)}\Big{\langle}
g,K_{\alpha}(\cdot,w)\Big{\rangle}_{\alpha}K_{\alpha}(\cdot,v) \\ & =
e^{-\frac{\alpha}{2}(|w|^2+|v|^2)} \left(K_{\alpha}(\cdot,v)\otimes K_{\alpha}(\cdot,w)\right) g.
\end{align*}
 Thus, we only need to show that operators of the form
$C_{\alpha}(v)\big{(} 1\otimes 1 \big{)}C_{\alpha}(-w)$ can be approximated by Toeplitz operators
with symbols in $C_b ^\infty$.

   Moreover, since \begin{align*}
C_{\alpha}(v)T_{q_{\beta}}^{\alpha}C_{\alpha}(-w)
&=C_{\alpha}(v)C_{\alpha}(-w)C_{\alpha}(w)T_{q_{\beta}}^{\alpha}C_{\alpha}(-w)\\ &=e^{i\alpha \cdot
\textup{Im}(v\cdot \overline{w})}C_{\alpha}(v-w)T_{q_{\beta}(\cdot-w)}^{\alpha}, \end{align*} we can
write $C_{\alpha}(v)T_{q_{\beta}}^{\alpha}C_{\alpha}(-w)$  as a single Toeplitz operator
$T_{{\mathcal{F}}_{\beta}}^{\alpha}$ with symbol (depending on $v$ and $w$)  ${\mathcal{F}}_{\beta}
\in C_b ^\infty$.   Finally, this fact tells us that \begin{align*} \underset{\beta \rightarrow
\infty}{\lim} \left\| T_{{\mathcal{F}}_{\beta} }^{\alpha}  \right. & - \left.C_{\alpha}(v)\big{(}
1\otimes 1\big{)}C_{\alpha}(-w)\right\|_{F_{\alpha}^p\rightarrow F_{\alpha}^p} \nonumber \\ & =
\underset{\beta \rightarrow \infty}{\lim} \left\| C_{\alpha}(v)T_{q_{\beta}}^{\alpha}C_{\alpha}(-w)
-C_{\alpha}(v)\big{(} 1\otimes 1\big{)}C_{\alpha}(-w)\right\|_{F_{\alpha}^p\rightarrow F_{\alpha}^p}
\nonumber \\ & \leq \underset{\beta \rightarrow \infty}{\lim}  \left\|
T_{q_{\beta}}^{\alpha}-1\otimes 1\right\|_{F_{\alpha}^p\rightarrow F_{\alpha}^p} \nonumber \\ & = 0
\end{align*} where the last equality follows from Lemma $3.8$. \end{proof}

\section{Sampling and interpolation results for the Fock space}

  The proofs of the following two lemmas (Lemmas $4.1$ and $4.3$) borrow deep ideas from the theory
  of sampling and interpolation in Fock spaces.  In particular, the proof of Lemma $4.1$ is similar
  to the proof of Theorem $5.1$ in \cite{MT}.  On the other hand, Lemma $4.3$ is a ``folklore"
  result in sampling theory and follows from the machinery developed in \cite{G} for abstract
  coorbit spaces.  However, since Lemma $4.3$ is not explicitly stated in \cite{G}, we will provide
  a short and direct proof.

Before we state and prove Lemma $4.1$, we need to briefly discuss the pseudo-hyperbolic metric
$\rho$ on $\mathbb{B}_n$.  Given any $z \in \mathbb{B}_n$, let $\phi_z$ be the involutive
automorphism of $\mathbb{B}_n$ that interchanges $0$ and $z$.  The  pseudo-hyperbolic metric $\rho$
on $\mathbb{B}_n$ is then defined by the formula \begin{equation*} \rho(z, w) = |\phi_z (w)|.
\end{equation*} It is well known (see \cite{Z}) that $\rho$ is indeed a metric on $\mathbb{B}_n$ and
that $\rho$ satisfies the identity \begin{equation*}  1 - (\rho(z, w))^2 = \frac{(1 - |z|^2)(1 -
|w|^2)}{|1 - z \cdot \overline{w}|}. \end{equation*}

\begin{lemma} Suppose that $1 < p \leq 2, \ r > 1,$ and $w_k \in B(0, r)$ for $k = 1, \ldots,  m$
are points where $|w_k - w_j| \geq \epsilon > 0$ if $j \neq k.$  Then for any $1 \leq k_0 \leq m,$
there exists $g_{k_0} \in F_\alpha ^p$ and a constant $C = C( \epsilon, r) > 0$ (which is assumed to
also depend on $n, \alpha,$ and $p$, but does not depend on the sequence $\{w_k\}_k$ itself ) such
that \begin{align} g_{k_0} (w_k) = \delta_{k_0, k} \text{  and  } \|g_{k_0}\|_{\alpha, p} \leq C.
\nonumber \end{align}   \end{lemma}

\begin{remark} Since we require that $C = C(\epsilon, r)$ does not depend on the actual sequence
$\{w_k\}_k$ itself, Lemma $4.1$ does not immediately follow from the results in \cite{MT}.
\end{remark}

\begin{proof}    First note that if $1 < p < 2$, $\alpha > 0$, and $g \in \Lp$, then a direct
application of H\"{o}lder's inequality tells us that $\|g\|_{\alpha, p} \leq C_{\alpha, \alpha'}
\|g\|_{\alpha', 2}$ for any $0 < \alpha' < \alpha$. Thus, it is enough to prove the lemma for $p =
2$ and arbitrary $\alpha > 0$.  For the rest of the proof, $C$ will denote a positive constant that
may depend on $\epsilon, r, n, p$ and $\alpha$, but not on the actual sequence $\{w_k\}_k$ itself.
Now if  $|w_k - w_j| \geq \epsilon $ when $j \neq k,$ then clearly $\{B(w_k, \frac{\epsilon}{2})
\}_k $ is a pairwise disjoint sequence of balls with \begin{align} \bigcup_k B\left(w_k,
\frac{\epsilon}{2}\right) \subset B\left(0, r + \frac{\epsilon}{2}\right) \nonumber \end{align}
which means that \begin{align} m \leq \left(\frac{2r}{\epsilon} + 1 \right)^{2n} =: M_{r, \epsilon}.
\nonumber \end{align}   Thus, since  \begin{align}  \underset{ k}{\inf} \prod_{j \neq k}^m
\frac{|w_j - w_k|}{r +  \epsilon} > \left( \frac{\epsilon }{r + \epsilon} \right) ^{M_{r, \epsilon}
- 1} \nonumber \end{align} for $k \leq m$,  it follows from the discussion preceeding the statement
of Lemma $4.1$ that \begin{align} \underset{k}{\inf} \prod_{j \neq k}^m  \rho \left(\frac{w_j}{r +
\epsilon} , \frac{w_k}{r +  \epsilon} \right) > C \nonumber \end{align}  for  $k \leq m$.  Now, it
is easy to construct a bounded function $\varphi_{k_0}  $ that is holomorphic on $B(0, r +
\epsilon)$ with \begin{align} \varphi_{k_0}   (w_k) = \delta_{k_0, k} \text{  and  } \underset{z \in
B(0, r + \epsilon)}{\sup} |\varphi_{k_0}  (z)|    \leq C. \nonumber \end{align}  In particular, let
\begin{align} \widetilde{\varphi_{k_0}}  (z) := \prod_{j \neq k_0}^m \frac{\overline{\phi_{w_j}
(w_{k_0})} \phi_{w_j} (z) }{| \phi_{w_j} (w_{k_0})|^2} \nonumber \end{align} and set $\varphi_{k_0}
(z)  := \widetilde{\varphi_{k_0}} \left(\frac{z}{r + \epsilon} \right)$.

Let $C_c^\infty (\C)$ denote the space of all smooth, compactly supported complex valued functions
on $\C$.  Pick any $\eta \in C_c ^\infty (\mathbb{C}^n)$ where \begin{align*} \eta \equiv 1 \text{
on } B\left(0, \frac{\epsilon}{2} \right) \text{ and } \eta \equiv 0 \text{ on }  \C \backslash
B\left(0, \frac{2\epsilon}{3} \right). \end{align*} If we define $\psi \in C_c ^\infty
(\mathbb{C}^n)$ by \begin{equation*} \psi (z) := \sum_{k = 1}^m \eta(z - w_k), \end{equation*} then
$\psi$ satisfies

\begin{align}  \psi \equiv 1  \text{ on } \bigcup_{k = 1}^m B\left(w_k, \frac{\epsilon}{2} \right)
\text{  and  }
 \ \psi \equiv 0 \text{ on } \mathbb{C}^n \backslash  \bigcup_{k = 1}^m B\left(w_k,
 \frac{2\epsilon}{3} \right). \nonumber \end{align}

\noindent If we extend $\varphi_{k_0}   (z) $ to $|z| \geq r + \epsilon$ by setting $\varphi_{k_0}
(z)  \equiv 0$ for $|z| \geq r + \epsilon$ and let $\widetilde{F_{k_0}} (z) = \psi (z) \varphi_{k_0}
(z)$, then $\widetilde{F_{k_0}} \in C_c^\infty (\mathbb{C}^n)$ satisfies

\begin{list}{}{\setlength\parsep{0in}} \item $(i) \ \widetilde{F_{k_0}} (w_k) = \delta_{k_0, k},$
 \item $(ii) \  \|\widetilde{F_{k_0}} \|_{L^\infty} \leq C,$
\item $(iii) \ \overline{\partial} \widetilde{F_{k_0}}  $ is supported on $\bigcup_{k = 1}^m B(w_k,
    \epsilon) \backslash B(w_k, \frac{\epsilon}{2}),$ \item $(iv) \ \widetilde{F_{k_0}} $ is
    supported on $\bigcup_{k = 1}^m B(w_k, \epsilon).$ \end{list}

Now for large $R > r $, let $v(z)$ be the negative function  \begin{align} v(z) := n \sum_{k : |z -
w_k| < R} \left[ \log \left| \frac{z- w_k}{R}\right|^2 + 1 - \left| \frac{z - w_k}{R} \right|^2
\right]. \nonumber \end{align}    It is easy to see that $\phi (z) := v(z) + \frac{\alpha}{2} |z|^2$
is plurisubharmonic for $R$  large enough (depending on $\epsilon$ and $r$), and so H\"{o}rmander's
Theorem (Theorem $4.4.2$ in \cite{H}) gives us a (distributional) solution $u \in L_{\text{loc}} ^2
(\mathbb{C}^n)$ to the equation $\overline{\partial} u = \overline{\partial}  \widetilde{F_{k_0}}$
where \begin{align} \int_{\mathbb{C}^n} |u(z)|^2 (1 + |z|^2)^{-2} e^{- \frac{\alpha}{2} |z|^2} \,
dv(z) & \leq \int_{\mathbb{C}^n} |u(z)|^2 (1 + |z|^2)^{-2} e^{- \phi (z) } \, dv(z) \nonumber \\ &
\leq \int_{\mathbb{C}^n} |\overline{\partial}  \widetilde{F_{k_0}} (z)|^2 e^{- \phi (z) } \, dv(z).
\tag{4.1} \end{align} However, since  \begin{align} \left| \phi (z) - \log |z - w_k|^{2n} -
\frac{\alpha}{2} |z|^2\right| \leq C \text{  when  } |z - w_k | < \epsilon, \nonumber \end{align} we
get that $|\phi(z)| \leq C$ for all $z \in \bigcup_{k = 1}^m B(w_k, \epsilon) \backslash B (w_k,
\frac{\epsilon}{2})$. Moreover, since $\overline{\partial} \widetilde{F_{k_0}}  $ is supported on
$\bigcup_{k = 1}^m B(w_k, \epsilon) \backslash B (w_k, \frac{\epsilon}{2})$, we get from $(4.1)$
that \begin{align} \int_{\mathbb{C}^n} |u(z)|^2  e^{- \alpha |z|^2} \, dv(z) & \leq C
\int_{\mathbb{C}^n} |u(z)|^2 (1 + |z|^2)^{-2} e^{- \frac{\alpha}{2} |z|^2} \, dv(z) \nonumber \\ &
\leq C \int_{\mathbb{C}^n} |\overline{\partial}  \widetilde{F_{k_0}} (z)|^2 e^{- \phi(z) } \, dv(z)
\nonumber \\ &  \leq C \underset{z \in \bigcup_{k = 1}^m B(w_k, \epsilon) \backslash B (w_k,
\frac{\epsilon}{2})}{\sup} |\overline{\partial}  \widetilde{F_{k_0}} (z)|^2 \tag{4.2} \\ &\leq C
\nonumber \end{align} where the last inequality follows from the product rule combined with the
Cauchy estimates applied to $\widetilde{F_{k_0}}$

Now note that if $F_{k_0} := u - \widetilde{F_{k_0}}$, then $F_{k_0}$ is entire, so that $u \in
C^\infty(\mathbb{C}^n)$ and  $\|F_{k_0}\|_{\alpha, 2} \leq C $.  Finally, $(4.2 )$ and the fact that
$e^{- \phi(z)} \approx |z - w_k|^{-2n} $ for $z$ near $w_k$ tells us that $u(w_k) = 0$, so that
$F_{k_0}(w_k) = \delta_{k_0, k}$, which completes the proof.

\end{proof}

We need to set up some simple notation and machinery before we state and prove Lemma $4.3$.  Let
$\mathbb{H}_n = \mathbb{C}^n \times \partial \mathbb{D}$ be the quotient of the $n$ dimensional
complex Heisenberg group by $2\pi \mathbb{Z}$, with group law \begin{align} (z_1, t_1)  (z_2, t_2) =
(z_1 + z_2, t_1 t_2 e^{- i \alpha \text{Im } {z_1} \cdot \overline{z_2}}) \nonumber \end{align} and
with Haar measure $m$ being the Lebesgue measure $m =  dv d\theta$ on $\mathbb{C}^n \times \partial
\mathbb{D}$ where $d\theta$ is the ordinary (normalized) arc length measure on $\partial
\mathbb{D}$.   Let  $T : F_\alpha ^p \rightarrow L^p(\mathbb{H}_n)$ be the isometry given by
\begin{align} Tf(z, t) = t e^{-\frac{\alpha}{2}  |z|^2} f(z). \nonumber \end{align}  For $f \in
L^p(\mathbb{H}_n)$ and $g \in L^q(\mathbb{H}_n)$ where $q$ is the dual exponent of $p$,  let $f \ast
g$ be the convolution product defined by \begin{align} f*g (h) = \int_{\mathbb{H}_n} f(y) g(h
y^{-1}) \, dm(y) \nonumber \end{align} for $h \in \mathbb{H}_n$.  If $G(z, u) = u
e^{-\frac{\alpha}{2} |z|^2}$, then the reproducing property of $F_\alpha ^p$ tells us that $F \ast G
= F$ for any $F \in T(F_\alpha ^p)$.

Now, enumerate $\epsilon  \mathbb{Z}^{2n}$ for fixed $\epsilon$ as $\{z_j\}_j$. For any fixed
integer $N_\epsilon > \epsilon^{-1}$ and any integer $0 \leq k < N_\epsilon$, let $u_k = \exp
(\frac{2\pi i k}{N_\epsilon})$.  Let $U_\epsilon = [0, \epsilon )^{2n}  \times \{e^{2\pi i \theta} :
0 \leq \theta < \frac{1}{N_\epsilon} \} \subset \mathbb{H}_n$ and (for any integer $0 \leq k <
N_\epsilon$) let $G_{j k}$ be the set $U_{\epsilon}$ translated on the right by $(z_j, u_k)$, so
that $G_{jk} = U_\epsilon  (z_j, u_k)$.  Clearly we then have that:

\begin{list}{}{\setlength\parsep{0in}} \item $(a) \ \mathbb{H}_n = \bigcup_{j, k} G_{j k},$ \item
$(b) \ G_{j k} \cap G_{j'k'} = \emptyset \text{ if } (j, k) \neq (j', k').$ \end{list}

Note that $m(G_{jk})$ only depends on $\epsilon$ and not on $j$ or $k$.  Thus, if $c_\epsilon =
m(G_{jk})$, then we can define an operator $R_\epsilon $ on $T(F_\alpha^p) \subset L^p
(\mathbb{H}_n)$ by \begin{align} R_\epsilon F (z, u) := c_\epsilon \sum_{j, k}  F(z_j, u_k) G((z, u)
(z_j, u_k)^{-1}). \nonumber \end{align} By a direct calculation we have that $R_\epsilon :
T(F_\alpha^p) \rightarrow T(F_\alpha^p)$ boundedly. In particular, if $f \in F_\alpha ^p$, then
\begin{align} R_\epsilon Tf (z, u) & = c' _\epsilon u e^{-\frac{\alpha}{2} |z|^2} \sum_{j} f(z_j)
e^{\alpha (z \cdot \overline{z_j}) - \alpha |z_j|^2}  \nonumber \\ & = c'_\epsilon  u e^{-
\frac{\alpha}{2} |z|^2} T_{\nu_\epsilon} f (z) \nonumber \end{align} where $c_\epsilon ' = v([0,
\epsilon )^{2n})$, $v$ is the ordinary Lebesgue volume measure on $\mathbb{C}^n$, and $\nu_\epsilon$
is the measure \begin{align} \nu_\epsilon = \sum_{\sigma \in \epsilon \mathbb{Z}^{2n} }
\delta_\sigma \nonumber \end{align} where $\delta_\sigma$ is the point-mass measure at $\sigma$. But
since $\nu_\epsilon$ is a Fock-Carleson measure, it is clear that $R_\epsilon :  T(F_\alpha^p)
\rightarrow T(F_\alpha^p)$ boundedly.   Let $\chi_{jk}$ be the characteristic function of $G_{jk}$
and define the operator $S_\epsilon : T(F_\alpha^p) \rightarrow L^p(\mathbb{H}_n)$ by \begin{align}
S_\epsilon F = \sum_{j, k} F(z_j, u_k) \chi_{jk} \ast G. \nonumber \end{align} Finally, define the
sharp maximal function $G^\sharp _{U_\epsilon}$ on $\mathbb{H}_n$ by \begin{align} G^\sharp
_{U_\epsilon} (h) = \underset{u \in U_\epsilon}{\sup} |G(u^{-1} h ) - G(h)| \nonumber \end{align}
and define $G^{\widetilde{\sharp}}  _{U_\epsilon}$ on $\mathbb{H}_n$ by \begin{align}
G^{\widetilde{\sharp}} _{U_\epsilon} (h) = \underset{u \in U_\epsilon}{\sup} |G( h u ) - G(h)|.
\nonumber \end{align}

\begin{lemma} Given $F \in T(F_\alpha^p)$, we have that \begin{align} \|  F - S_\epsilon
F\|_{L^p(\mathbb{H}_n)} \leq o(\epsilon) \|F\|_{L^p(\mathbb{H}_n)} \nonumber \end{align}  where
$\underset{\epsilon \rightarrow 0^+}{\lim} \ o(\epsilon) = 0.$ \end{lemma}

\begin{proof} Since $F \ast G = F$ for $F \in T(F_\alpha ^p)$, Young's convolution inequality gives
us that  \begin{align} \|F - S_\epsilon F \|_{L^p(\mathbb{H}_n)}  & = \|(F  - \sum_{j, k} F(z_j,
u_k)  \chi_{jk} ) \ast G \|_{L^p(\mathbb{H}_n)} \nonumber \\ & \leq  \|F - \sum_{j, k} F(z_j, u_k)
\chi_{jk} \|_{L^p(\mathbb{H}_n)} \|G\|_{L^1(\mathbb{H}_n)}. \nonumber \end{align} We now  estimate
$F - \sum_{j, k} F(z_j, u_k)  \chi_{jk}$ pointwise using the reproducing property as follows:
\begin{align} |F(z, u) & - \sum_{j, k} F(z_j, u_k)  \chi_{jk} (z, u)| \nonumber \\ & \leq \sum_{j,k}
|F(z, u) - F(z_j, u_k)| \chi_{jk} (z, u) \nonumber \\ & = \sum_{j,k} \left| \int_{\mathbb{H}_n}
[G((z, u) y^{-1}) - G((z_j, u_k) y^{-1})] F(y) \, dm(y) \right|  \chi_{jk} (z, u). \tag{4.3}
\end{align} Fix any $j$ and $k$.  Since $(z, u) \in G_{jk}$ for each summand in $(4.3)$, we can
write $(z, u) = (z', u') (z_j, u_k)$ where $(z', u') \in U_\epsilon$, so that $(z_j, u_k) = (z',
u')^{-1} (z, u)$.  Plugging this into $(4.3)$ and recalling the definition of $G^\sharp
_{U_\epsilon}$, we get that \begin{align}  |F(z, u) & - \sum_{j, k}  F(z_j, u_k)  \chi_{jk} (z, u)|
\nonumber \\ & \leq  \sum_{j,k} \left| \int_{\mathbb{H}_n} [G((z, u) y^{-1}) - G((z', u')^{-1} (z,
u) y^{-1})] F(y) \, dm(y) \right| \chi_{jk} (z, u) \nonumber \\ & \leq   \sum_{j,k}
\left(\int_{\mathbb{H}_n} G^\sharp _{U_\epsilon} ((z, u) y^{-1}) |F(y)| \, dm(y) \right) \chi_{jk}
(z, u) \nonumber \\ & = |F| \ast G^\sharp _{U_\epsilon}  (z, u). \nonumber \end{align}  The proof is
now completed by another application of Young's convolution inequality and the easily checked fact
that \begin{align} \lim_{\epsilon \rightarrow 0^+} \|G^\sharp _{U_\epsilon} \|_{L^1(\mathbb{H}_n)} =
0. \nonumber \end{align}\end{proof}

Now we will state and prove Lemma $4.3$. Note that in the language of sampling theory, Lemma $4.3$
states that the ``frame operator'' \begin{equation*} f \mapsto \sum_{\sigma \in \epsilon
\mathbb{Z}^{2n}} \langle f, k_\sigma ^\alpha \rangle_\alpha  k_\sigma ^\alpha \end{equation*} on
$F_\alpha ^p$ associated to the frame $\{k_\sigma ^\alpha\}_{\sigma \in \epsilon \mathbb{Z}^{2n}}$
for small enough $\epsilon > 0$ is invertible. \begin{lemma}  Let $1 < p < \infty$ and let
$\nu_\epsilon$ be the measure \begin{align} \nu_\epsilon = \sum_{\sigma \in \epsilon \mathbb{Z}^{2n}
} \delta_\sigma \nonumber \end{align}  where $\delta_\sigma$ is the point-mass measure at $\sigma$.
Then $T_{\nu_\epsilon}$ is invertible on $F_\alpha ^p$ for small enough $\epsilon > 0$.
\end{lemma}

\begin{proof} Note that by definition, we have $T_{\nu_\epsilon} f   = c_\epsilon '  T^{-1}
R_\epsilon T f$. Thus, it is enough to show that $R_\epsilon :   T(F_\alpha^p) \rightarrow
T(F_\alpha^p) $ is invertible. By Lemma $4.2$, this will be proved if we can show that \begin{align}
\|R_\epsilon F - S_\epsilon F\|_{L^p(\mathbb{H}_n)} \leq o(\epsilon) \|F\|_{L^p(\mathbb{H}_n)}
\nonumber \end{align} where $F \in T(F_\alpha ^p)$ and $\underset{\epsilon \rightarrow 0^+}{\lim} \
o(\epsilon) = 0.$

To that end, we now pointwise estimate $|R_\epsilon F - S_\epsilon F|$ as follows: \begin{align}
|R_\epsilon   F   (z, u) & -  S_\epsilon F (z, u) | \nonumber \\ &  \leq  \sum_{j, k} |(m(G_{jk})
F(z_j, u_k) G( (z, u)(z_j, u_k)^{-1}) \nonumber \\ &  \hspace{47mm} -   F(z_j, u_k) \int_{G_{jk}}
G((z, u) y^{-1}) \, dm(y)) | \nonumber \\ & \leq   \sum_{j, k} |F(z_j, u_k)| \int_{G_{jk}} |G((z, u)
(z_j, u_k)^{-1}) - G((z, u) y^{-1})| \, dm(y). \tag{4.4} \end{align} However, since $y \in G_{jk}$
we can write $y = y' (z_j, u_k)  $ for some $y' \in U_\epsilon$ so that $(z_j, u_k)^{-1} =  y^{-1}
y'$, and plugging this into $(4.4)$ gives us that \begin{align} |R_\epsilon F (z, u) - & S_\epsilon
F (z, u) | \nonumber \\ & \leq \sum_{j, k} |F(z_j, u_k)| \int_{G_{jk}} |G((z, u) y^{-1} y') - G((z,
u) y^{-1})| \, dm(y) \nonumber \\ & \leq  \sum_{j, k} |F(z_j, u_k)| \int_{G_{jk}}
|G^{\widetilde{\sharp}} _{U_\epsilon} (z, u) y^{-1})| \, dm(y)  \nonumber \\ & = \left(\sum_{j,k}
|F(z_j, u_k)| \chi_{jk} \right) \ast G^{\widetilde{\sharp}} _{U_\epsilon}. \nonumber \end{align}
Again, it is easy to see that \begin{align} \lim_{\epsilon \rightarrow 0^+} \|G^{\widetilde{\sharp}}
_{U_\epsilon} \|_{L^1(\mathbb{H}_n)} = 0 \nonumber \end{align} so by Young's inequality, we only
need to show that  \begin{align} \|\sum_{j,k} |F(z_j, u_k)| \chi_{jk}\|_{L^p(\mathbb{H}_n)} \leq C
\|F\|_{L^p(\mathbb{H}_n)}  \nonumber \end{align}  which follows easily from Lemma $2.4$.
\end{proof}

For the rest of the paper, $\nu$ will denote the Fock-Carleson measure $\nu_{\epsilon_0}$ from Lemma
$4.3$ where $\epsilon_0	 $ is fixed and small enough so that $T_{\nu} ^\alpha$ is invertible.

\section{A uniform algebra $\mathcal{A}$ and its maximal ideal space}

Let $\mathcal{A}\subset L^{\infty}$ be the unital $C^*$-algebra of all bounded and uniformly
continuous functions on $\mathbb{C}^n$. Since $C_b^{\infty}\subset \mathcal{A}$, it follows from the
Theorem $3.7$ that $\mathcal{T}_p ^\alpha$ for $1< p<\infty$ is the closed algebra generated by Toeplitz operators with symbols in $\mathcal{A}$.  In this
section, we will extend the Berezin transform and other related objects defined on $\mathbb{C}^n$ to
$M_\mathcal{A}$, where $M_{\mathcal{A}}$ denotes the space of non-zero multiplicative functionals on
$\mathcal{A}$ equipped with the weak${}^*$ topology. Note that $\mathcal{A}$ is not separable and
hence the space $M_{\mathcal{A}}$ is not metrizable. Since $\mathcal{A}$ is a commutative, unital
$C^*$-algebra, the Gelfand-transform $\wedge :\mathcal{A} \rightarrow C(M_{\mathcal{A}})$ defined by
$\hat{a}(\varphi)=\varphi(a)$ for $a\in \mathcal{A}$ and $\varphi \in M_{\mathcal{A}}$ gives us an
isomorphism between $\mathcal{A}$ and $C(M_{\mathcal{A}})$. In the following we will often write
$a(\varphi)$ instead of $\hat{a}(\varphi)$. For $x \in \mathbb{C}^n$, let  $\delta_x \in
M_\mathcal{A}$ be the point evaluation at $x$ defined by $\delta_x(f)=f(x)$. It is not difficult to
see that the map $x \mapsto \delta_x$  induces a dense embedding of $\mathbb{C}^n$ into
$M_{\mathcal{A}}$.

For $w \in \mathbb{C}^n$, let $\tau_w$ be the usual translation function $\tau_w(z):=z-w$. More
generally, if $x \in M_{\mathcal{A}}$ and $w\in \mathbb{C}^n$, then define $\tau_x\in \prod_{w \in
\mathbb{C}^n} M_{\mathcal{A}}$ by $\tau_x(w)(a):=x(a\circ \tau_w)$ where $a\in \mathcal{A}$. We will
write $a\circ \tau_x(w)$ instead of $\tau_x(w)(a)$ since $\tau_x$ naturally extends the translation
by elements in $\mathbb{C}^n$ to a ``translation'' by elements in $M_{\mathcal{A}}$. Let $\epsilon
>0$, $w_1,w_2\in \mathbb{C}^n$ and $a\in \mathcal{A}$, then we have \begin{equation*} \left|a\circ
\tau_x(w_1)-a \circ\tau_x(w_2)\right|\leq \| a\circ \tau_{w_1}-a\circ \tau_{w_2}\|_\infty < \epsilon
\end{equation*} if $|w_1-w_2|<\delta$ where $\delta>0$ is chosen suitably according to the uniformly
continuity of $a$. Therefore we have shown that the map $\tau_x: \mathbb{C}^n\rightarrow
M_{\mathcal{A}}$ is continuous. Next, prove:
\begin{lemma}\label{lemma_uniformly_compact_convergence} If $(z_{\beta})_{\beta}$ is a net in
$\mathbb{C}^n$ converging to $x \in M_{\mathcal{A}}$, then $a\circ \tau_{z_{\beta}}(w)\rightarrow
a\circ \tau_x(w)$ for all $a\in \mathcal{A}$ and $w\in \mathbb{C}^n$ where the convergence is
uniform on compact subsets of $\mathbb{C}^n$. \end{lemma} \begin{proof} Since $a\circ \tau_w \in
\mathcal{A}$ for all $w \in \mathbb{C}^n$, it follows by definition of the convergence $z_{\beta}
\rightarrow x$ in $M_{\mathcal{A}}$ that $a\circ \tau_{z_{\beta}}(w)=\delta_{z_{\beta}}(a\circ
\tau_w) \rightarrow x(a\circ \tau_w)=a\circ \tau_x(w)$ for all $a\in \mathcal{A}$ and $w \in
\mathbb{C}^n$. Assume that the above convergence is not uniform on compact subsets of
$\mathbb{C}^n$. Then there is $\epsilon >0,$ a function $a\in \mathcal{A},$ and a compact set $K \in
\mathbb{C}^n$ such that for all $\gamma,$ there is $\beta >\gamma$ and $\xi_{\beta}\in K$ with
\begin{equation}\label{estimate_from_above} \left| (a\circ \tau_{z_{\beta}})(\xi_{\beta})-(a\circ
\tau_x)(\xi_{\beta})\right|>\varepsilon. \tag{5.1} \end{equation} By passing to a subnet, we can
assume that $\xi_{\beta}\rightarrow \xi \in K$. Now, we have \begin{align*} | (a\circ
\tau_{z_{\beta}})(\xi_{\beta})& -(a\circ \tau_x)(\xi_{\beta})|\leq \left|(a\circ
\tau_{z_{\beta}})(\xi_{\beta})-(a\circ \tau_{z_{\beta}})(\xi)\right| \nonumber \\ & +\left|(a\circ
\tau_{z_{\beta}})(\xi)-(a\circ \tau_x)(\xi) \right|+ \left|(a\circ \tau_x)(\xi) -(a\circ
\tau_x)(\xi_{\beta})\right|. \end{align*} Since
$|\tau_{z_{\beta}}(\xi_{\beta})-\tau_{z_{\beta}}(\xi)|=|\xi_{\beta}-\xi|$ it follows that the first
term on the right hand side tends to zero. The third term on the right tends to zero by the
continuity of $\tau_x: \mathbb{C}^n\rightarrow M_{\mathcal{A}}$, and the second term tends to zero
by what was said at the beginning of the proof. We obtain a contradiction to $(5.1)$ and the lemma
is proven. \end{proof}

For $w \in \mathbb{C}^n$, let $C_\alpha (w)$ be the
``weighted shifts'' defined in Section $3$.  If $1\leq p\leq \infty$ and $A$ is a fixed bounded
operator on $F_{\alpha}^p$, then we write $A_w:=C_{\alpha}(w)AC_{\alpha}(-w)$,  which induces a map
$\Psi_A: \mathbb{C}^n \rightarrow \mathcal{L}(F_{\alpha}^p)$ defined by $\Psi_A(w):=A_w$. Since
\begin{equation*} \big{[}C_{\alpha}(-w)K_{\alpha}(\cdot ,
\xi)\big{]}(z)=K_{\alpha}(z,\xi-w)K_{\alpha}(w,\xi)e^{-\frac{\alpha}{2}|w|^2}, \end{equation*} we
have that \begin{equation} B_\alpha \circ \Psi_A(w)=B_\alpha (A_w)=B_\alpha (A)\circ \tau_w.
\tag{5.2} \end{equation}

Let $E$ be a metric space and let $f: \mathbb{C}^n\rightarrow E$. Consider the (possibly empty)
multi-valued function on $M_{\mathcal{A}}$ defined by \begin{equation*} F(x_0):=\big{\{} \lambda :
f(z_{\beta})\rightarrow \lambda : \mbox{for some net } \; z_{\beta}\rightarrow x_0\: \mbox{and } \;
z_{\beta}\in \mathbb{C}^n\big{\}} \end{equation*} where $x_0 \in M_{\mathcal{A}}$. We will say that
$F$ is ``single valued'' if for any $x_0\in M_{\mathcal{A}}$ and any convergent net
$z_{\beta}\rightarrow x_0$ where $z_{\beta} \subset \C$, the net $(f(z_{\beta}))_{\beta}$ converges in $E$.

\begin{lemma} Assume that $F$ is single valued for all $x_0\in M_{\mathcal{A}}$. Then
$F:M_{\mathcal{A}}\rightarrow E$ is well-defined and continuous. \end{lemma} \begin{proof} Since $\C$ is dense in $\mathcal{A}$ and $E$ (as a metric space) is regular, the result follows immediately by our hypotheses and Bourbaki's extension theorem (Theorem $1$, p. 81 in \cite{B}). \end{proof}

 \noindent Let  $ B_1(F_\alpha ^p)$ denote the unit ball (in the norm topology)
of the space of bounded operators on $F_\alpha ^p$.  If ``SOT'' refers to the strong operator
topology, then recall that $(B_1 (F_\alpha ^p), \SOT)$ is a complete metric space since $F_\alpha
^p$ is separable. In the next proposition we will fix $A \in \mathcal{T}_p ^\alpha$ and without loss of
generality assume that $A \in  B_1(F_\alpha ^p)$. \begin{proposition} If $1 \leq p < \infty$ and
$A\in \mathcal{T}_p ^\alpha$, then $\Psi_A:\mathbb{C}^n\rightarrow (B_1(F_\alpha ^p), \SOT)$ extends
continuously to $M_{\mathcal{A}}$. \end{proposition} \begin{proof} Consider the multi-valued
function \begin{equation*} {\bf \Psi}_A(x):=\big{\{} \lambda \: : \: \Psi_A(z_{\gamma})\rightarrow
\lambda \mbox{ for some net } \: z_{\gamma}\rightarrow x \mbox{ and } z_{\gamma}\in \mathbb{C}^n
\big{\}} \end{equation*} where $x\in M_{\mathcal{A}}$. According to Lemma $5.2$ we need to show that
${\bf \Psi}_A(x)$ is single valued. Since $(B_1(F_\alpha ^p), \SOT)$ is a complete metric space, it
is sufficient to show that $\{\Psi_A (z_{\gamma})\}_{\gamma}$ is a Cauchy net whenever
$\{z_{\gamma}\}_{\gamma}  \subset \mathbb{C}^n$ is a net converging to some $x \in M_\mathcal{A}$.

To that end, let $\{z_{\gamma}\}_{\gamma}  \subset \mathbb{C}^n$ be a net that converges to
$x\in M_{\mathcal{A}}$.  Let $A\in \mathcal{T}_p ^\alpha$ and pick $\epsilon >0$. Choose $R$ in the
(non-closed) algebra generated by $\{T_f ^\alpha : f \in \mathcal{A}\}$ with $\|R\|_{F_\alpha ^p
\rightarrow F_\alpha ^p} \leq 1$ such that $\|A-R\|_{F_{\alpha}^p\rightarrow
F_{\alpha}^p}<\epsilon$.  Then for all $f\in F_{\alpha}^p$, we have that \begin{align*}
\big{\|}\big{[} & \Psi_A  (z_{\gamma})-\Psi_A(z_{\beta})\big{]}f\big{\|}_{F_p^{\alpha}\rightarrow
F_p^{\alpha}} \\ &\leq \Big{[}\|A_{z_{\gamma}}-R_{z_{\gamma}}\|_{F_{\alpha}^p\rightarrow
F_{\alpha}^p} +\|R_{z_{\beta}}-A_{z_{\beta}}\|_{F_{\alpha}^p\rightarrow
F_{\alpha}^p}\Big{]}\|f\|_{\alpha,p}+\|[ R_{z_{\gamma}}-R_{z_{\beta}}]f\|_{\alpha,p}\\ &\leq
2\epsilon \|f\|_{\alpha,p} + \big{\|}\big{[}
\Psi_R(z_{\gamma})-\Psi_R(z_{\beta})\big{]}f\big{\|}_{\alpha,p}. \end{align*} Therefore, it is
sufficient to show that $\{\Psi_R(z_{\gamma})\}_{\gamma}$ is a Cauchy net with respect to the SOT
for all $R$ in the algebra generated by $\{T_f ^\alpha : f \in \mathcal{A}\}$ with $\|R\|_{F_\alpha
^p \rightarrow F_\alpha ^p} \leq 1$. Moreover, by linearity and Proposition $3.6$, we can assume
that $R=T^{\alpha}_{a_1}T^{\alpha}_{a_2}\cdots T^{\alpha}_{a_m}$ is a finite product of Toeplitz
operators with symbols in $a_j\in \mathcal{A}$ where $\|a_j\|_{\infty} \leq 1$. Since the product of
convergent  nets in $(B_1(F_\alpha ^p), \SOT)$ is convergent, is it sufficient to show that $\{
\Psi_{T_a^{\alpha}}(z_{\gamma})\}_{\gamma}$ for all $a\in \mathcal{A}$ with $\|a\|_{\infty} \leq 1$
has a limit in $(B_1(F_\alpha ^p), \SOT)$. Now if $w\in \mathbb{C}^n$, then \begin{align*}
B\big{(}T_{a\circ \tau_w}^{\alpha}\big{)}(z) & =\widetilde{(a\circ \tau_w)}^{(\frac{1}{4\alpha})}(z)
=\widetilde{a}^{(\frac{1}{4\alpha})}\circ \tau_w(z) \nonumber \\ &
=B\big{(}T_{a}^{\alpha}\big{)}\circ \tau_w(z) =B\big{(}(T_{a}^{\alpha})_w\big{)}(z) \end{align*}
where we have used standard properties of the heat transform together with $(5.2)$ in the last
equality. Since the Berezin transform is one-to-one on operators, it follows that \begin{equation*}
\Psi_{T_a^{\alpha}}(z_{\gamma})=\left( T_a^{\alpha}\right)_{z_{\gamma}}=T_{a\circ
\tau_{z_{\gamma}}}^{\alpha}. \end{equation*} Let $f \in \mathcal{A}$ and let
$(f_{\beta})_{\beta}\subset \mathcal{A}$ be a net that converges to $f$ uniformly on compact subsets
of $\mathbb{C}^n$. Then it is easy to check that $T_{f_{\beta}}^{\alpha}\rightarrow T_f^{\alpha}$ in
SOT. According to Lemma \ref{lemma_uniformly_compact_convergence}, we have the uniform compact
convergence $a\circ \tau_{z_{\gamma}}\rightarrow a\circ\tau_x$, and therefore
$\Psi_{T^{\alpha}_a}(z_{\gamma}) =T_{a\circ \tau_{z_{\gamma}}}^{\alpha}\rightarrow
T^{\alpha}_{a\circ \tau_x}$ in SOT. \end{proof}

The proof of the following corollary follows precisely from the proof of Proposition $5.3$, though
it will be useful to record it for future use. \begin{corollary}Suppose that $1 <  p < \infty$ and $S \in \mathcal{T}_p ^\alpha$. If $x \in M_\mathcal{A} \backslash
\mathbb{C}^n$ and $\{z_\gamma\}_\gamma$ is a net converging to $x$, then $ S_x = \lim_\gamma
S_{z_\gamma}$ where the limit is taken in the strong operator topology. \end{corollary}

\section{Proof of Theorem $1.1$} Finally in this section we will give a proof of Theorem $1.1$. As
in \cite{S}, the proof will follow easily from quantitative estimates on the essential norm
$\|A\|_{\e}$ for operators $A \in \mathcal{T}_p ^\alpha$.  We will first prove the following simple lemma
and then state some definitions that are needed for the proof of Theorem $1.1$.

\begin{lemma} Let $A\in \mathcal{T}_p ^\alpha$, then $\textup{(a)}$ and $\textup{(b)}$ below are equivalent:
\begin{itemize} \item[\textup{(a)}] $B_\alpha (A)(z)\rightarrow 0$ as $|z|\rightarrow \infty$.
\item[\textup{(b)}] $A_x=0$ for all $x\in M_{\mathcal{A}}\setminus \mathbb{C}^n$. \end{itemize}
\end{lemma}

\begin{proof} Obviously we may assume that $\|A\|_{F_\alpha ^p \rightarrow F_\alpha ^p} \leq 1$.
$\textup{(a)}\Rightarrow \textup{(b)}$: Let $x\in M_{\mathcal{A}}\setminus \mathbb{C}^n$ and
$(z_{\gamma})_{\gamma}$ be a net with $z_{\gamma}\rightarrow x$. According to Corollary $5.4$, we
have for all fixed $\xi \in \mathbb{C}^n$ that: \begin{align} \left| B_\alpha
(A_{z_{\gamma}})(\xi)-B_\alpha (A_x)(\xi)\right| & =\left|\left\langle
\big{[}A_{z_{\gamma}}-A_x\big{]} k_{\xi}^{\alpha},k_{\xi}^{\alpha}\right\rangle_\alpha \right|
\nonumber \\ & \leq \left\|\left[
A_{z_{\gamma}}-A_x\right]k^{\alpha}_\xi\right\|_{\alpha,p}\stackrel{\gamma}{\longrightarrow} 0.
\tag{6.1} \end{align} \noindent Combining $(6.1)$ with $(5.2)$ tells us that \begin{align*}
\left|B_\alpha (A)(\xi-z_{\gamma})\right|=\left|B_\alpha
(A_{z_{\gamma}})(\xi)\right|\stackrel{\gamma}{\longrightarrow} \left|B_\alpha (A_x)(\xi)\right|.
\end{align*} Since $x\in M_{\mathcal{A}}\setminus \mathbb{C}^n$, we can assume that
$|z_{\gamma}|\rightarrow \infty$ and from the condition $\lim_{|z|\rightarrow \infty}B_\alpha
(A)(z)=0$ we conclude that $B_\alpha (A_x)(\xi)=0$ for all $\xi \in \mathbb{C}^n$. Since $B_\alpha$ is
one-to-one on the bounded operators on $F_{\alpha}^p$ it follows that $A_x=0$. \vspace{1ex}\\
$\textup{(b)}\Rightarrow \textup{(a)}$: Assume that there is a sequence $(z_k)_k\subset\mathbb{C}^n$
such that $|z_k|\rightarrow \infty$ and \begin{equation} |B_\alpha (A_{z_k})(0)|=|B_\alpha
(A)(-z_k)|\geq \delta >0. \tag{6.2} \end{equation} \noindent Since $M_{\mathcal{A}}$ is compact there
is a subnet $(z_{\gamma})_{\gamma}$ of $(z_k)_k$ and $x\in M_{\mathcal{A}}$ such that
$z_{\gamma}\rightarrow x \in M_{\mathcal{A}}$. From $(6.1)$ with $\xi = 0$ and $(6.2)$, we get that
$|B_\alpha (A_x)(0)|\geq \delta >0$, which says that $A_x\ne 0$ as desired. \end{proof}

 Now for any bounded operator $S$ on $F_\alpha ^p$ and any $r > 0$, let \begin{align} \alpha_S (r)
 := \underset{|z| \rightarrow \infty}{\limsup } \ \sup \{ \|Sf\|_{\alpha, p} : f \in T_{\chi_{B(z,
 r)} \nu} ^\alpha (F_\alpha ^p), \|f\|_{\alpha, p} \leq 1 \}. \nonumber \end{align}   An easy application of Lemma $4.1$ tells us that $T_{\chi_{B(z, r_1)} \nu} ^\alpha (F_\alpha ^p) \subset
 T_{\chi_{B(z, r_2)} \nu} ^\alpha   (F_\alpha ^p)$ when $r_1 < r_2$, which means that $\alpha_S (r)$ is an increasing function of $r$. In particular, recall from the end of Section $4$ that  \begin{equation*} \nu = \sum_{\sigma \in \epsilon_0 \mathbb{Z}^{2n} } \delta_\sigma \end{equation*} for $\epsilon_0$ fixed so that $T_\nu ^\alpha$ is invertible.  Now if $h \in T_{\chi_{B(z, r_1)} \nu} ^\alpha (F_\alpha ^p)$ then \begin{equation*} h(w) = \sum_{\sigma \in \epsilon_0 \mathbb{Z}^{2n} \cap B(z, r_1)} g(\sigma) e^{\alpha (w \cdot \overline{\sigma}) - \alpha |\sigma|^2} \end{equation*} for some $g \in F_\alpha ^p$.  For each $\sigma \in \epsilon_0 \mathbb{Z}^{2n} \cap B(z, r_1)$, Lemma $4.1$ allows us to pick some $g_\sigma \in F_\alpha ^p$ where $g_\sigma (\sigma') = \delta_{\sigma, \sigma'}$ for any $\sigma' \in \epsilon_0 \mathbb{Z}^{2n} \cap B(z, r_2)$.   Thus, if \begin{equation*} \widetilde{g} := \sum_{ \sigma \in \epsilon_0 \mathbb{Z}^{2n} \cap B(z, r_1)} g(\sigma) g_\sigma, \end{equation*} then $\widetilde{g} \in F_\alpha ^p$ and clearly $h = T_{\chi_{B(z, r_2)} \nu} ^\alpha \widetilde{g}$. Note that since $\alpha_S (r) \leq \|S\|$ for all $r$, we have \begin{align} \alpha_S  := \underset{r \rightarrow \infty}{\lim} \alpha_S (r) = \underset{r >
 0}{\sup} \ \alpha_S (r) \leq \|S\|. \nonumber \end{align} 	

The proof of Theorem $1.1$ will follow easily from Lemma $6.1$ and the following result, Theorem
$6.2$.  In the statement and proof of this result, we will use the symbol $``\approx"$ to indicate
that two quantities are equivalent with constants only depending on $\alpha, \epsilon_0, p$, and
$n$. Moreover, $C$ will denote a constant depending only on $\alpha, \epsilon_0, p, $ and $n$, and
can possibly change from line to line. \begin{theorem} Let $2 \leq p < \infty$ and let $A \in
\mathcal{T}_p ^\alpha.$ If $\|A\|_{\e}$ denotes the essential norm of $A$, then $\|A\|_{\e}$ is equivalent
to the following quantities (with constants depending on only $\alpha, \epsilon_0, p, $ and $n$)

\begin{list}{}{\setlength\parsep{0in}} \item $(i) \ \alpha_A$,
 \item $(ii) \  \beta_A := \sup_{d> 0} \  \lim \sup_{|z| \rightarrow \infty} \|M_{\chi_{B(z, d)}} A
     \|_{F_\alpha ^p \rightarrow L_\alpha ^p },$
\item $(iii) \  \gamma_A := \lim_{r \rightarrow \infty}  \|M_{\chi_{B(0, r) ^c} } A \|_{F_\alpha ^p
    \rightarrow L_\alpha ^p}$ where $B(0, r) ^c = \mathbb{C}^n \backslash B(0, r) $. \end{list}
    Moreover, for all $1 < p < \infty$, we have that \begin{equation} \|A\|_{\e} \approx \underset{x
    \in M_{\mathcal{A}} \backslash \mathbb{C}^n}{\sup} \|A_x\|_{F_\alpha ^p \rightarrow F_\alpha
    ^p}. \nonumber \end{equation}

\end{theorem}

\begin{remark} Note that the proof of Theorem $6.2$ is similar to the proofs of Theorems $5.2$ and
$9.3$ in \cite{S}.  Thus, we will sometimes only outline arguments of the proof and refer the reader
to \cite{S} for the full details. \end{remark}

\begin{proof}    First note that in the last statement of Theorem $6.2$, it is enough to prove
\begin{equation} \|A\|_{\e} \approx \underset{x \in M_{\mathcal{A}} \backslash \mathbb{C}^n}{\sup}
\|A_x\|_{F_\alpha ^p \rightarrow F_\alpha ^p} \nonumber \end{equation} if $A \in \mathcal{T}_p ^\alpha$ for
$2 \leq p < \infty$, since if $1 < p \leq 2$,  $q$ is the dual exponent, and $A \in \mathcal{T}_p ^\alpha$,
then \begin{align} \|A\|_{\e} & = \|A^*\|_{\e} \nonumber \\ & \approx \underset{x \in
M_{\mathcal{A}} \backslash \mathbb{C}^n}{\sup} \|(A^*) _x\|_{F_\alpha ^q \rightarrow F_\alpha ^q}
\nonumber \\ & = \underset{x \in M_{\mathcal{A}} \backslash \mathbb{C}^n}{\sup} \|(A
_x)^*\|_{F_\alpha ^q \rightarrow F_\alpha ^q} \nonumber \\ & = \underset{x \in M_{\mathcal{A}}
\backslash \mathbb{C}^n}{\sup} \|A _x\|_{F_\alpha ^p \rightarrow F_\alpha ^p} \nonumber \end{align}
where we have used the equality $(A^*) _x = (A_x)^* $ for all $x \in M_\mathcal{A}$, which follows
from the $\WOT$ continuity of taking adjoints combined with Corollary $5.4$.

As in \cite{S}, we will use the notation $\|\cdot\|_{\e}$ and $\|\cdot\|_{\text{ex}}$ to distinguish
the essential norms of an operator from $F_\alpha ^p$ to itself and an operator from $F_\alpha ^p$
to $L_\alpha ^p$.  Moreover, if $R$ is a bounded operator from $F_\alpha ^p$ to itself, then it is
easy to see that \begin{equation} \|R\|_{\text{ex}} \leq \|R\|_{\e} \leq \|P_\alpha \|_{L_\alpha ^p
\rightarrow L_\alpha ^p} \|R\|_{\text{ex}} \nonumber \end{equation} so that $\|R\|_{\text{ex}} $ and
$\|R\|_{\e}$ are equivalent.

Now pick $\epsilon > 0$ and choose Borel sets $F_j \subset G_j \subset \mathbb{C}^n$ as in Lemma
$3.3$ where  \begin{equation} \|A T_\nu ^\alpha - \sum_j M_{\chi_{F_j}} A T_{ \chi_{G_j} \nu}
^\alpha\|_{F_\alpha ^p \rightarrow L_\alpha ^p} \leq \epsilon  \tag{6.3} \end{equation} and write
$A_m = \sum_{j \geq m} M_{\chi_{F_j}} A T_{\chi_{G_j} \nu} ^\alpha $.  Since $\sum_{j = 1}^m
M_{\chi_{F_j}} A T_{\chi_{G_j}  \nu} ^\alpha$ is compact for any $m \geq 1$, $(6.3)$ tells us that
\begin{equation} \|A T_\nu ^\alpha - A_m\|_{\text{ex}} < \epsilon \tag{6.4} \end{equation} for any
$m \geq 1$.  However, by Lemma $3.3$, the monotonicity of the function $r \mapsto \alpha_A (r)$, and
the arguments in \cite{S} p. 2209-2210, we have that \begin{equation*} \limsup_{m \rightarrow
\infty} \|A_m\|_{F_\alpha ^p \rightarrow L_\alpha ^p} \leq C \alpha_A \end{equation*}  which
combined with $(6.4)$ tells us that \begin{align*} \|A T_\nu ^\alpha\|_{\text{ex}}  \leq \limsup_{m
\rightarrow \infty} \|A_m\|_{F_\alpha ^p \rightarrow L_\alpha ^p} + \epsilon \nonumber \leq C
\alpha_A + \epsilon. \nonumber \end{align*} But since Lemma $4.3$ tells us that $T_\nu ^\alpha$ is
invertible, letting $\epsilon \downarrow 0$ in the previous inequality tells us that\begin{equation}
\|A\|_{\e} \leq C \alpha_A. \tag{6.5}  \end{equation}

Now again by the arguments in \cite{S}, p. 2210 - 2211, we have that $\beta_A, \gamma_A$, and
$\limsup_{m \rightarrow \infty} \|A_m\|_{F_\alpha ^p \rightarrow L_\alpha ^p}$ are equivalent.
Moreover $(6.5)$ will give us that $\|A\|_{\e}$ is equivalent to $(i), (ii), $ and $(iii)$ if we can
show that $\alpha_A \leq C \|A\|_{\e}$.

If $Q$ is some compact operator from $F_\alpha ^p$ to itself, then it follows from Theorem $3.9$ that $Q \in \mathcal{T}_p ^\alpha$,  and thus it follows easily from
Lemma $6.1$ that $Q_x = 0$ for all $x \in M_\mathcal{A} \backslash \mathbb{C}^n$. Fix $x \in
M_\mathcal{A} \backslash \mathbb{C}^n$ and let $\{z_\gamma\}_\gamma \subset \mathbb{C}^n$ be a net
converging to $x$.     Since $\SOT$ limits (except for a multiplicative constant) do not increase
the norm, Proposition $5.3$ and the $\SOT$ convergence $A_{z_\gamma} + Q_{z_\gamma} \rightarrow A_x
+ Q_x = A_x$ gives us that \begin{equation*} \|A_x\|_{F_\alpha ^p \rightarrow F_\alpha ^p} \leq C
\liminf_\gamma \|A_{z_\gamma} + Q_{z_\gamma}\|_{F_\alpha ^p \rightarrow F_\alpha ^p}.
\end{equation*} Since this holds for all $x \in M_\mathcal{A} \backslash \mathbb{C}^n$ and all
compact $Q$, we easily get that \begin{equation} \underset{x \in M_{\mathcal{A}} \backslash
\mathbb{C}^n}{\sup} \|A_x\|_{F_\alpha ^p \rightarrow F_\alpha ^p} \leq C \|A\|_{\e}. \tag{6.6}
\nonumber \end{equation}

Finally, by $(6.5)$ and $(6.6)$, the proof will be completed if we can show that \begin{equation*}
\alpha _A \leq C \underset{x \in M_{\mathcal{A}} \backslash \mathbb{C}^n}{\sup} \|A_x\|_{F_\alpha ^p
\rightarrow F_\alpha ^p}.  \end{equation*} \noindent To that end, let $r > 0$. Pick a sequence
$\{z_j\}_j$ tending to $\infty$ as $ j \rightarrow \infty$ and a normalized sequence $f_j \in
T_{\chi_{B(z_j, r)} \nu} ^\alpha (F_\alpha ^p) $ such that $\|A f_j\|_{\alpha, p} \rightarrow
\alpha_A (r)$.  Thus, there are $h_j \in F_\alpha ^p$ where \begin{equation*} f_j(w) =
T_{\chi_{B(z_j, r)} \nu} ^\alpha h_j (w)  = \sum_{\sigma \in \epsilon_0  \mathbb{Z}^{2n} \cap B(z_j
, r) } h_j (\sigma) e^{\alpha (w \cdot \overline{\sigma}) - \alpha |\sigma|^2}. \nonumber
\end{equation*} For each $ \sigma \in \epsilon_0  \mathbb{Z}^{2n} \cap B(z_j , r)$, let $\sigma(j) =
\sigma - z_j$ so that each $\sigma(j) \in B(0, r)$.  A direction computation now tells us that
\begin{equation*} C_\alpha (z_j)^* f_j(w) = \sum_{\sigma(j) \in B(0, r)} a_{\sigma(j)} k_{\sigma(j)}
^\alpha(w) \nonumber \end{equation*} where $a_{\sigma(j)} = h_j(\sigma) e^{-\frac{\alpha}{2}
|\sigma|^2} e^{ i \alpha \Imag \sigma \cdot \overline{z_j}} $. Let $q$ be the dual exponent of $p$,
so that $1 < q \leq 2$.   Now for each fixed $j$ and fixed $\sigma_0 (j) \in B(0, r)$, pick $g =
g_{j, \sigma_0(j)}$ according to Lemma $4.1$ where \begin{displaymath}
   g(\sigma(j)) = \left\{
     \begin{array}{lr}
       1 & \text{ if } \sigma = \sigma_0\\
       0 &  \text{ if } \sigma \neq \sigma_0
     \end{array}
   \right.
\end{displaymath} and $\|g_{j, \sigma_0 (j)}\|_{\alpha, q} \leq C$ where $C = C(\alpha, \epsilon_0,
p, n, r)$.  Then by the reproducing property, we get that \begin{align*} \langle C_\alpha (-z_j)
f_j, g \rangle _{F_\alpha ^2} & = \sum_{\sigma(j) \in B(0, r)} a_{\sigma(j)} e^{-\frac{\alpha}{2}
|\sigma(j)|^2} g(\sigma(j)) \nonumber \\ & = a_{ \sigma_0 (j)} e^{- \frac{\alpha}{2} |\sigma_0(j)|
^2 } \nonumber \end{align*} which says that each $ |a_{\sigma(j)}| \leq C$ for some $C = C(\alpha,
\epsilon_0, p, n, r)$.

Now pick some $M = M(\epsilon_0, n , r)$ where $ M'(j):= \card \epsilon_0  \mathbb{Z}^{2n} \cap
B(z_j, r)  \leq M$ and enumerate $\epsilon_0  \mathbb{Z}^{2n} \cap B(z_j, r)$ as $\sigma_1, \ldots,
\sigma_M$. Clearly we may choose a subsequence of $\{z_j\}_j$ such that $M'(j) \equiv M_0 \leq M$
for some $M_0 \in \mathbb{N}$ that is independent of $j,$ and so without loss of generality we can
assume that \begin{equation*}\{(\sigma_1(j), \ldots, \sigma_{M_0}(j), a_{\sigma_1(j)}, \ldots,
a_{\sigma_{M_0}(j)})\}_{j = 1}^\infty \subset \mathbb{C}^{(n+1)M_0}. \end{equation*}  Since the
sequence $\{(\sigma_1(j), \ldots, \sigma_{M_0}(j), a_{\sigma_1(j)}, \ldots,
a_{\sigma_{M_0}(j)})\}_{j = 1}^\infty$ is bounded, we can (passing to another subsequence if
necessary) assume that this sequence converges to a point \begin{equation*} (\sigma_1, \ldots,
\sigma_{M_0}, a_{\sigma_1}, \ldots, a_{\sigma_{M_0}}) \in \mathbb{C}^{(n+1)M_0} \end{equation*}
where each $|\sigma_i| \leq  C= C(\alpha, \epsilon_0, p, n, r)$.  An easy application of the
Lebesgue dominated convergence theorem then gives us that \begin{equation} C_\alpha (z_j) ^* f_j
\rightarrow h := \sum_{i = 1}^{M_0} a_{\sigma_{i} } k_{\sigma_i} \text{ in } F_\alpha ^p
\tag{6.7}\end{equation} which says that \begin{align} \alpha_A (r) = \lim_j \|A f_j \|_{\alpha, p }
& = \lim_j \|A_{- z_j} C_\alpha (-z_j) f_j \|_{\alpha, p} \nonumber \\ & = \lim_j \|A_{- z_j} h
\|_{\alpha, p}. \tag{6.8}  \end{align} Now since $M_\mathcal{A}$ is compact, we can choose a subnet
$\{- z_\gamma\}$ of the sequence $\{- z_j\}$ converging to some $x \in M_{\mathcal{A}} \backslash
\mathbb{C}^n$, which means that $\lim_\gamma \|A_{-z_\gamma} h\|_{\alpha, p} = \lim_\gamma \|A_{-x}
h\|_{\alpha, p}$.  Combining this with $(6.8)$ and Proposition $5.3$ finally gives us that
\begin{align*} \alpha_A (r)  = \lim_\gamma \|A_{- z_\gamma} h \|_{\alpha,  p} = \|A_x h \|_{\alpha,
p}  \leq  \underset{u \in M_{\mathcal{A}} \backslash \mathbb{C}^n}{\sup} \|A_u\|_{F_\alpha ^p
\rightarrow F_\alpha ^p} \nonumber \end{align*} since $(6.7) $ tells us that $\| h\|_{\alpha, p} =
1$.  \end{proof}

\begin{corollary} Let $1 < p < \infty$ and $A \in \mathcal{T}_p ^\alpha$.  Then \begin{equation*} \|A\|_{\e}
\approx \sup_{\|f\|_{\alpha , p} = 1} \limsup _{|z| \rightarrow \infty} \|A_z f \|_{\alpha, p}.
\end{equation*} \end{corollary} \begin{proof}

The proof is identical to the proof of Corollary $9.4$ in \cite{S}. \end{proof}

\noindent \textit{Proof of Theorem $1.1$.} It has already been proven that $A$ is compact on
$F_\alpha ^p$ only if $A \in \mathcal{T}_p ^\alpha$ and it was remarked in the introduction that $B_\alpha
(A)$ vanishes at infinity if $A$ is compact.  Now if $B_\alpha (A)$ vanishes at infinity, then Lemma
$6.1$ tells us that $A_x = 0$ for all $x \in M_\mathcal{A} \backslash \mathbb{C}^n$.  Theorem $6.2$
then says that $\|A\|_{\e} = 0$, which means that $A$ is compact.  $ \hfill \square$

\section{The case $p = 2$ and open problems} As in \cite{S}, we can significantly improve Theorem
$6.2$ when $p = 2$. If $A \in F_\alpha ^2$, let $\sigma(A)$ be the spectrum of $A$ and let $r(A) =
\max \{|\lambda| : \lambda \in \sigma(A)\}$.  Moreover, let $\sigma_{\e}(\sigma)$ be the essential
spectrum of $A$ and let  $r_{\e} (A) = \max \{|\lambda| : \lambda \in \sigma_{\e} (A)\}.$

\begin{theorem}

If $A \in \mathcal{T}_2 ^\alpha$, then \begin{equation*} \|A\|_{\e} = \underset{x \in M_\mathcal{A}
\backslash \mathbb{C}^n}{\sup} \|A_x \|_{F_\alpha ^2 \rightarrow F_\alpha ^2}. \nonumber
\end{equation*}  Moreover, \begin{equation*}  \underset{x \in M_\mathcal{A} \backslash
\mathbb{C}^n}{\sup} r(A_x) \leq \lim_{k \rightarrow \infty} \left(\underset{x \in M_\mathcal{A}
\backslash \mathbb{C}^n}{\sup} \|A_x ^k \|_{F_\alpha ^2 \rightarrow F_\alpha ^2} ^\frac{1}{k}
\right) = r_{\e} (A)   \nonumber \end{equation*} with equality if $A$ is essentially normal.

\end{theorem}

\begin{theorem} If $A \in \mathcal{T}_2 ^\alpha$ then following are equivalent:
\begin{list}{}{\setlength\parsep{0in}} \item $(i) \ \lambda \notin \sigma_{\e} (A) $,
 \item $(ii) \  \lambda \notin \bigcup_{x \in M_\mathcal{A} \backslash \mathbb{C}^n} \sigma (A_x)
     \text{ and } \underset{x \in M_\mathcal{A} \backslash \mathbb{C}^n }{\sup} \|(S_x - \lambda I )
     ^{-1} \|_{F_\alpha ^2 \rightarrow F_\alpha ^2} < \infty$
\item $(iii) \  \text{there is } \gamma > 0 \text{ depending only on } \lambda, \text{ such that } $
    \begin{equation*}    \|(S_x - \lambda I ) f  \|_{\alpha, 2} \geq \gamma \|f\|_{\alpha, 2} \text{
    and } \|(S_x ^* - \overline{\lambda} I ) f  \|_{\alpha, 2} \geq \gamma \|f\|_{\alpha, 2}
    \end{equation*} for all $f \in F_\alpha ^2$ and $x \in M_{\mathcal{A}} \backslash \mathbb{C}^n$.
    \end{list} \end{theorem}

\begin{theorem}
 If $A \in \mathcal{T}_2 ^\alpha$, then \begin{equation*} \overline{\bigcup_{x \in M_{\mathcal{A}}
 \backslash \mathbb{C}^n} \sigma (A_x) } \subset \sigma_{\e} (A) \end{equation*} \end{theorem}

The proofs of these results are identical to the proofs of the corresponding results in \cite{S},
and in particular depend on the well known fact that \begin{equation*} r(A) = \lim_{k \rightarrow
\infty} \|A^k \|_{F_\alpha ^2 \rightarrow F_\alpha ^2} ^\frac{1}{k} = \|A\|_{F_\alpha ^2 \rightarrow
F_\alpha ^2} \end{equation*} whenever $A$ is self adjoint.  \\

 We will close this paper with a discussion of some open problems.  By using ideas in \cite{G, MT}, it is very likely that the results in Section $4$ hold for more
general weighted Fock spaces $F_\phi ^p$ for suitable weight functions $\phi : \mathbb{C}^n
\rightarrow \mathbb{R}^+$, where \begin{equation*} F_\phi ^p:=  \{f \text{ entire} : f(\cdot)  e^{-
\phi (\cdot)} \in L^p (\mathbb{C}^n, dv)\}. \end{equation*} Thus, it would be interesting to know if
our results hold for other weighted Fock spaces.

For the space $F_\alpha ^2$, one should notice that the reproducing kernel and the Gaussian weight behave extremely ``nicely'' together.  One should also notice that this simple fact was crucial in proving many of the results in Sections $2$ and $3$ (and is in fact crucial for proving compactness results for individual Toeplitz operators or finite sums of finite products of Toeplitz operators on the Fock space, see \cite{BCI, I}.) Unfortunately, such nice behavior between the reproducing kernel and more general weights rarely holds, and overcoming this would most likely be the most challenging obstacle in extending the results of this paper to more general weighted Fock spaces.

Note that one can easily find examples of functions $f \in L_\alpha ^1$ where $B_\alpha (f)$ vanishes at infinity but $T_f ^\alpha$ is not compact on $F_\alpha ^2$ (see \cite{BC, BCI}).  It would be interesting to know if other such examples can be found for $F_\alpha ^p$ when $p \neq 2$.  However, it would be far more interesting to know if one could come up with similar examples for the Bergman space of the ball, weighted or unweighted (where the condition $f \in L_\alpha ^1$ is replaced by $f \in L^1(\mathbb{B}_n, dv_\gamma)$ in the weighted case.)

Finally, it would be interesting to know whether the results of this paper hold for the $p = 1$ or $p = \infty$ case.  While some incomplete results are known in the Bergman space setting when $p = 1$ or $p = \infty$ (see \cite{GL, Y}), it appears that there are no known results for the Fock space when $p = 1$ or $p = \infty$.

\section*{Acknowledgements} The authors would like to thank Karlheinz Gr\"{o}chenig and Kristian
Seip for their useful information regarding Section $4$. The authors would also like to thank the reviewer for his or her helpful comments that greatly improved the exposition, and would like to thank Brett Wick for interesting comments regarding this paper and for providing us with the preprint \cite{MW}.

\end{document}